\documentclass[10pt]{article}
\usepackage{latexsym,amsfonts,amssymb,amsmath,amsthm}
\usepackage{graphicx}

\usepackage[usenames,dvipsnames]{color}
\usepackage{ulem}

\begin{document}
\setlength{\baselineskip}{16pt}

\parindent 0.5cm
\evensidemargin 0cm \oddsidemargin 0cm \topmargin 0cm \textheight
22cm \textwidth 16cm \footskip 2cm \headsep 0cm

\newtheorem{theorem}{Theorem}[section]
\newtheorem{lemma}{Lemma}[section]
\newtheorem{proposition}{Proposition}[section]
\newtheorem{definition}{Definition}[section]
\newtheorem{example}{Example}[section]
\newtheorem{corollary}{Corollary}[section]

\newtheorem{remark}{Remark}[section]

\numberwithin{equation}{section}

\def\p{\partial}
\def\I{\textit}
\def\R{\mathbb R}
\def\C{\mathbb C}
\def\u{\underline}
\def\l{\lambda}
\def\a{\alpha}
\def\O{\Omega}
\def\e{\epsilon}
\def\ls{\lambda^*}
\def\D{\displaystyle}
\def\wyx{ \frac{w(y,t)}{w(x,t)}}
\def\imp{\Rightarrow}
\def\tE{\tilde E}
\def\tX{\tilde X}
\def\tH{\tilde H}
\def\tu{\tilde u}
\def\d{\mathcal D}
\def\aa{\mathcal A}
\def\DH{\mathcal D(\tH)}
\def\bE{\bar E}
\def\bH{\bar H}
\def\M{\mathcal M}
\renewcommand{\labelenumi}{(\arabic{enumi})}

\def\disp{\displaystyle}
\def\undertex#1{$\underline{\hbox{#1}}$}
\def\card{\mathop{\hbox{card}}}
\def\sgn{\mathop{\hbox{sgn}}}
\def\exp{\mathop{\hbox{exp}}}
\def\OFP{(\Omega,{\cal F},\PP)}
\newcommand\JM{Mierczy\'nski}
\newcommand\RR{\ensuremath{\mathbb{R}}}
\newcommand\CC{\ensuremath{\mathbb{C}}}
\newcommand\QQ{\ensuremath{\mathbb{Q}}}
\newcommand\ZZ{\ensuremath{\mathbb{Z}}}
\newcommand\NN{\ensuremath{\mathbb{N}}}
\newcommand\PP{\ensuremath{\mathbb{P}}}
\newcommand\abs[1]{\ensuremath{\lvert#1\rvert}}

\newcommand\normf[1]{\ensuremath{\lVert#1\rVert_{f}}}
\newcommand\normfRb[1]{\ensuremath{\lVert#1\rVert_{f,R_b}}}
\newcommand\normfRbone[1]{\ensuremath{\lVert#1\rVert_{f, R_{b_1}}}}
\newcommand\normfRbtwo[1]{\ensuremath{\lVert#1\rVert_{f,R_{b_2}}}}
\newcommand\normtwo[1]{\ensuremath{\lVert#1\rVert_{2}}}
\newcommand\norminfty[1]{\ensuremath{\lVert#1\rVert_{\infty}}}

\title{Positive Stationary Solutions and Spreading Speeds of KPP Equations
in Locally Spatially Inhomogeneous Media\thanks{Partially supported by NSF grant DMS--0907752}}

\author{
Liang Kong and Wenxian Shen \\
Department of Mathematics and Statistics\\
Auburn University\\
Auburn University, AL 36849\\
U.S.A. }

\date{}
\maketitle

\noindent{\bf Abstract.} The current paper is concerned with positive stationary solutions and
spatial spreading speeds of KPP type evolution equations with random or nonlocal or discrete dispersal
in locally spatially inhomogeneous media. It is shown that such an equation has a unique globally stable
positive stationary solution and has a spreading speed in every direction. Moreover, it is shown
that the  localized spatial inhomogeneity of the medium neither slows down nor speeds up the spatial spreading
in all the directions.

\medskip


\medskip


\section{Introduction}

The current paper is devoted to the study of spatial spreading dynamics of species in locally spatially inhomogeneous environments or media.
Reaction diffusion equations of the form
\vspace{-0.05in}\begin{equation}
\label{main-eq}
u_t(t,x)=\Delta u(t,x)+ u(t,x)f_1(x,u(t,x)),\quad x\in\RR^N
\vspace{-0.05in}\end{equation}
are widely used to model the population dynamics of many species
in unbounded environments, where $u(t,x)$ is the population density of the
species at time $t$ and location $x$, $\Delta u$ characterizes the
internal interaction of the organisms, and $f_1(x,u)$ represents the
growth rate of the population, which satisfies that $f_1(x,u)<0$ for $u\gg
1$ and $\p_uf_{1}(x,u)<0$ for $u\geq 0$ (see \cite{ArWe1}, \cite{ArWe2}, \cite{CaCo},  \cite{Fif1}, \cite{FiPe}, \cite{Fisher}, \cite{KPP},
\cite{Mur},  \cite{ShKa},
\cite{Ske}, \cite{Wei1}, \cite{Wei2}, \cite{Zha2}, etc.).

 When using \eqref{main-eq} to model the population dynamics of a species,
 it is assumed that the underlying  environment is not patchy and the internal interaction  of the organisms is
 random and local
(i.e.
the organisms move randomly between the adjacent spatial locations).
 In practice, the environments in which many species live may be patchy and/or the internal interaction of the organisms may be nonlocal.
To model the population dynamics of a species  in the case that the underlying  environment is not patchy  but
the internal interaction is nonlocal,
the following nonlocal dispersal equation is often used,
\vspace{-0.05in}\begin{equation}
\label{nonlocal-eq}
u_t(t,x)=\int_{\RR^N} \kappa(y-x)u(t,y)dy-u(t,x)+u(t,x)f_2(x,u(t,x)),\quad x\in\RR^N,
\vspace{-0.05in}\end{equation}
where  $\kappa(\cdot)$ is a smooth
 convolution kernel supported on a ball centered at the origin (that is, there is  a $\delta_0>0$ such that
  $\kappa(z)>0$ if $\|z\|<\delta_0$, $\kappa(z)=0$ if $\|z\|\geq \delta_0$, where
 $\|\cdot\|$ denotes the norm
 in $\RR^N$ and $\delta_0$ represents the nonlocal dispersal distance),  $\int_{\RR^N}\kappa(z)dz=1$, and
 $f_2(\cdot,\cdot)$ is of the same property  as $f_1$ in \eqref{main-eq}
 (see \cite{BaZh},  \cite{ChChRo}, \cite{CoCoElMa}, \cite{CoDyLa}, \cite{Fif2}, \cite{GrHiHuMiVi}, \cite{HuMaMiVi}, \cite{KaLoSh}, \cite{Lee}, \cite{Levin},
 etc.).
 Spatially discrete
dispersal equations of the following form arise when modeling the population dynamics of species living in  patchy environments,
\vspace{-0.08in}\begin{equation}
\label{discrete-eq}
u_t(t,j)=\sum_{k\in K}a_k(u(t,j+k)-u(t,j))+u(t,j)f_3(j,u(t,j)),\quad j\in\ZZ^N,
\vspace{-0.08in}\end{equation}
where $K=\{k\in\ZZ^N\,|\, \|k\|=1\}$, $a_k (k\in K)$ are positive constants,  and $f_3(j,u)<0$ for $u\gg
1$ and $\p_uf_{3}(j,u)<0$ for $u\geq 0$ (see \cite{Fif1},  \cite{MaWuZo}, \cite{Mur},  \cite{ShKa}, \cite{ShSw}, \cite{Wei1}, \cite{Wei2}, etc.).

Spatial spreading dynamics is one of the central dynamical issues of \eqref{main-eq}-\eqref{discrete-eq}.
Roughly speaking, it is about
how fast the population
spreads as time  evolves.  E.g., letting $\mathcal{H}=\RR^N$ in the case \eqref{main-eq} and \eqref{nonlocal-eq} and
$\mathcal{H}=\ZZ^N$ in the case of \eqref{discrete-eq}, $\xi\in S^{N-1}:=\{\xi\in\RR^N|\,\|\xi\|=1\}$,  and a given
 initial population $u_0$  satisfy  for some $\sigma_0>0$ that
 $u_0(x)\geq \sigma_0$ for  $x\in\mathcal{H}$ with $x\cdot\xi\ll -1$ and
 $u_0(x)=0$ for
 $x\in\mathcal{H}$ with $x\cdot\xi\gg 1$ ($x\cdot\xi$ is the inner product of $x$ and $\xi$), how fast
does the population invade into the region  with no population initially?

Since the pioneering works by Fisher \cite{Fisher} and Kolmogorov, Petrowsky, Piscunov
\cite{KPP} on the following special case of \eqref{main-eq}
\vspace{-0.05in}\begin{equation}
\label{classical-fisher-eq}
 u_t(t,x)=u_{xx}(t,x)+u(t,x)(1-u(t,x)),\quad\quad x\in \RR,
\vspace{-0.05in}\end{equation}
a vast amount research has been carried out toward  the spatial spreading dynamics of \eqref{main-eq}-\eqref{discrete-eq}
with $f_i(\cdot,\cdot)$ ($i=1,2,3$) being periodic in the
space variable,
 which reflects the spatial periodicity of the media. See, for example,  \cite{ArWe1}, \cite{ArWe2}, \cite{BeHaNa1}, \cite{BeHaNa2}, \cite{BeHaRo}, \cite{FrGa},
    \cite{Ham}, \cite{HuZi1}, \cite{Kam}, \cite{LiZh1},
    \cite{LiZh2}, \cite{LiYiZh}, \cite{Nad}, \cite{NoRuXi}, \cite{NoXi}, \cite{NoRoRyZl}, \cite{Sat}, \cite{Uch}, \cite{Wei1}, \cite{Wei2}, etc.
    for the study of \eqref{main-eq} in the case that $f_1(x,u)$ is periodic in $x$,
 see \cite{Cov}, \cite{CoDu}, \cite{CoDaMa}, \cite{HeShZh}, \cite{LiSuWa}, \cite{ShZh1}, \cite{ShZh2}, \cite{ShZh3}, etc.
 for the study of \eqref{nonlocal-eq} in the case that $f_2(x,u)$ is periodic in $x$, and see \cite{ChFuGu},  \cite{ChGu1}, \cite{ChGu2},  \cite{GuHa},
  \cite{GuWu1}, \cite{GuWu2}, \cite{HuZi2}, \cite{Lui}, \cite{Wei1}, \cite{Wei2}, \cite{ZiHaHu}, etc.  for the study of \eqref{discrete-eq}
  in the case that $f_3(j,u)$ is periodic in $j$. In such cases, the spatial spreading dynamics is quite well understood.
 For example, consider \eqref{main-eq} and assume that $f_1(x+p_i{\bf e_i}, u)=f_1(x,u)$ for $i=1,2,\cdots,N$, where
 $p_i$ ($i=1,2,\cdots, N$) are positive constants and
 \vspace{-0.05in}$${\bf e_i}=(\delta_{i1},\delta_{i2},
 \cdots,\delta_{iN}), \,\, \delta_{ij}=1\,\, {\rm if} \,\, i=j\,\, {\rm and}\,\, 0\,\, {\rm if}\,\, i\not =j.
 \vspace{-0.05in}$$
 If
 the principal eigenvalue of the following eigenvalue problem associated to the linearized equation of \eqref{main-eq} at $u=0$,
 \vspace{-0.05in}\begin{equation}
 \begin{cases}
 \Delta u(x)+f_1(x,0) u(x)=\lambda u(x),\quad x\in\RR^N\cr
 u(x+p_i{\bf e_i})=u(x),\quad x\in\RR^N,
 \end{cases}
 \vspace{-0.05in}\end{equation}
 is positive, then \eqref{main-eq} has a unique positive stationary solution $u_1^*(\cdot)$
  with $u_1^*(\cdot+p_i{\bf e_i})=u_1^*(\cdot)$ and for any $\xi\in S^{N-1}:=\{\xi\in\RR^N\,|\,\|\xi\|=1\}$, \eqref{main-eq}
 has a spreading speed $c_1^*(\xi)$ in the direction of $\xi$ in the following sense (see Definition \ref{spreading-def}
 for detail): for any given bounded $u_0\in C(\RR^N,\RR^+)$ with $\liminf_{x\cdot\xi \to -\infty}u_0(x)>0$ and
 $u_0(x)=0$ for $x\cdot\xi\gg 1$,
\vspace{-0.05in} $$
\liminf_{t\to\infty}\inf_{x\cdot\xi\leq ct}u_1(t,x;u_0)>0\quad\forall c<c_1^*(\xi)
\vspace{-0.05in}$$
and
\vspace{-0.05in}$$
\limsup_{t\to\infty}\sup_{x\cdot\xi\geq ct}u_1(t,x;u_0)=0\quad \forall c>c_1^*(\xi),
\vspace{-0.05in}$$
where $u_1(t,x;u_0)$ denotes the solution of \eqref{main-eq} with
$u_1(0,x;u_0)=u_0(x)$. Observe that \eqref{main-eq} has also
traveling wave solutions which connect $u_1^*(\cdot)$ and $0$ and
propagate in the direction of $\xi$ with  speeds greater than or
equal $c_1^*(\xi)$ and there is no such traveling wave solution of
slower speed (see \cite{BeHaRo}, \cite{LiZh2}, \cite{She1}, \cite{Wei2} for the definition of
spatially periodic traveling wave
solutions). Hence $c_1^*(\xi)$ is also the minimal wave speed of
traveling wave solutions propagating in the direction of $\xi$. See
\cite{BeHaRo}, \cite{HuZi1}, \cite{LiZh2}, \cite{Wei2} for the above mentioned results for \eqref{main-eq} and
see \cite{ShZh1}, \cite{ShZh2}, \cite{ShZh3} for similar results for \eqref{nonlocal-eq} and
\cite{GuWu1}, \cite{GuWu2}, \cite{HuZi2}, \cite{LiZh2}, \cite{Wei2},  \cite{ZiHaHu} for similar results for \eqref{discrete-eq}.

 In the current paper, we consider \eqref{main-eq}-\eqref{discrete-eq}
 in the case that the growth rates depend on the space variable, but only when it is in some bounded
subset of the underlying media, which reflects the localized spatial
inhomogeneity of the media. More precisely, let
\vspace{-0.05in}\begin{equation}
\label{habitat-eq}
\begin{cases}
\mathcal{H}_1=\mathcal{H}_2=\RR^N\cr
\mathcal{H}_3=\ZZ^N.
\end{cases}
\vspace{-0.05in}\end{equation}
We assume

\smallskip

\noindent{\bf (H1)} {\it $f_i:\mathcal{H}_i\times\RR\to\RR$  is a
$C^2$ function, $f_i(x,u)<0$ for all
$(x,u)\in\mathcal{H}_i\times\RR^+$ with $u\ge \beta_0$ for some $\beta_0>0$, and  $\p_uf_i(x,u)<0$ for
all $(x,u)\in\mathcal{H}_i\times\RR^+$, where $i=1,2,3$.}

\smallskip

\noindent{\bf (H2)} {\it  $f_i(x,u)=f_i^0(u)$ for some
$C^2$ function $f_i^0:\RR\to\RR$ and all $(x,u)\in\mathcal{H}_i\times
\RR$ with $\|x\|\ge L_0$ for some $L_0>0$, and $f_i^0(0)>0$, where $i=1,2,3$.}

\smallskip

Assume (H1) and (H2).
  Then
\eqref{main-eq}, \eqref{nonlocal-eq}, and \eqref{discrete-eq} have the following limit equations as
$\|x\|\to\infty$ or $\|j\|\to\infty$,
\vspace{-0.05in}\begin{equation}
\label{limit-main-eq} u_t(t,x)=\Delta u(t,x)+u(t,x) f_1^0(u(t,x)),\quad x\in\RR^N,
\vspace{-0.05in}\end{equation}
\vspace{-0.05in}\begin{equation}
\label{limit-nonlocal-eq}
u_t(t,x)=\int_{\RR^N}\kappa(y-x)u(t,y)dy-u(t,x)+u(t,x)f_2^0(u(t,x)),\quad x\in\RR^N,
\vspace{-0.05in}\end{equation}
and
\vspace{-0.05in}\begin{equation}
\label{limit-discrete-eq} u_t(t,j)=\sum_{k\in
K}a_k(u(t,j+k)-u(t,j))+u(t,j)f_3^0(u(t,j)),\quad j\in\ZZ^N.
\vspace{-0.05in}\end{equation}
Equations \eqref{limit-main-eq}, \eqref{limit-nonlocal-eq}, and \eqref{limit-discrete-eq}
will  play an important role in the study of \eqref{main-eq},
\eqref{nonlocal-eq}, and \eqref{discrete-eq}. Clearly,
\eqref{limit-main-eq} has similar spatial spreading dynamics as
that of \eqref{classical-fisher-eq}, that is, it
 has a unique positive constant solution $u_1^{0}$ and
 has a spatial spreading speed $c^{0}_1(\xi)$ in the direction of $\xi$
for every $\xi\in S^{N-1}$. Equations \eqref{limit-nonlocal-eq} (resp. \eqref{limit-discrete-eq}) has similar properties as  that of \eqref{limit-main-eq},
that is, \eqref{limit-nonlocal-eq} (resp. \eqref{limit-discrete-eq}) has a unique positive constant stationary solution
$u_2^0$ (resp. $u_3^0$) and has a spatial spreading speed $c^{0}_2(\xi)$ (resp. $c^0_3(\xi)$) in the direction of $\xi$
for every $\xi\in S^{N-1}$ (see Definition \ref{spreading-def} for detail).

Our objective is to explore   the
spatial spreading dynamics of \eqref{main-eq}-\eqref{discrete-eq} with localized spatial inhomogeneity.
The main results of this paper can be summarized as follows:

\smallskip

\noindent $\bullet$ {\it Assume ${\rm (H1)}$ and ${\rm (H2)}$. Then \eqref{main-eq}
$($resp. \eqref{nonlocal-eq}, \eqref{discrete-eq}$)$ has a unique
 positive stationary solution $u^*_1\in C(\RR^N, \RR^+)$
$($resp. $u_2^*\in C(\RR^N,\RR^+)$, $u_3^*\in C(\ZZ^N,\RR^+))$
satisfying that $\inf_{x\in\RR^N} u_1^*(x)>0$ $($resp.
$\inf_{x\in\RR^N}u_2^*(x)>0$, $\inf_{j\in\ZZ^N}u_3^*(j))$ and
$\lim_{\|x\|\to \infty} u_1^*(x)=u_1^{0}$ $($resp.
$\lim_{\|x\|\to\infty}u_2^*(x)=u_2^0$,
$\lim_{\|j\|\to\infty}u_3^*(j)=u_3^0)$. Moreover, $u=u_i^*(\cdot)$ is  globally asymptotically
stable with respect to positive perturbations $($and hence $u\equiv 0$ is an unstable stationary
solution of $(1.i))$ $(i=1,2,3)$ } (see Theorem
\ref{positive-solution-thm}).

\smallskip

\noindent $\bullet$  {\it Assume ${\rm (H1)}$ and ${\rm (H2)}$. Then \eqref{main-eq}
$($resp. \eqref{nonlocal-eq}, \eqref{discrete-eq}$)$ has a spatial
spreading speed $c^{*}_1(\xi)$ $($resp. $c_2^*(\xi)$,  $c^*_3(\xi))$
in the direction of $\xi$ for every $\xi\in S^{N-1}$ $($see
Definition \ref{spreading-def} for the definition of spreading
speeds$)$. Moreover, $c_1^*(\xi)=c_1^0(\xi)$ $($resp.
$c_2^*(\xi)=c_2^0(\xi)$, $c_3^*(\xi)=c_3^0(\xi))$ for all $\xi\in
S^{N-1}$, where $c_1^0(\xi)$ $($resp. $c_2^0(\xi)$, $c_3^0(\xi))$ is
the spatial spreading speed of \eqref{limit-main-eq} $($resp.
\eqref{limit-nonlocal-eq}, \eqref{limit-discrete-eq}$)$ in the
direction of $\xi$} (see Theorem \ref{spreading-speed-thm}).

\smallskip

\noindent $\bullet$ {\it Assume ${\rm (H1)}$ and ${\rm (H2)}$. Then the solution of
\eqref{main-eq} $($resp. \eqref{nonlocal-eq}, \eqref{discrete-eq}$)$
with a nonnegative initial data which has a nonempty compact set
spreads neither slower than $\inf\{c_1^*(\xi)|\xi\in S^{N-1}\}$
$($resp $\inf\{c_2^*(\xi)|\xi\in S^{N-1}\}$,
$\inf\{c_3^*(\xi)|\xi\in S^{N-1}\})$  nor faster than
$\sup\{c_1^*(\xi)|\xi\in S^{N-1}\}$ $($resp $\sup\{c_2^*(\xi)|\xi\in
S^{N-1}\}$, $\sup\{c_3^*(\xi)|\xi\in S^{N-1}\})$ } (see Theorem
\ref{spreading-speed-thm1} for detail).

\smallskip

The above results reveal such an important biological scenario: the localized spatial inhomogeneity  of the
media does not prevent the population to persist and  to spread,  moreover, it
neither slows down nor
speeds up the spatial spread of the population.

It should be pointed out that the authors of \cite{NoRoRyZl} considered the transition fronts, which are  generalizations of traveling wave solutions,
 of \eqref{main-eq} in the case that
$N=1$,  $f(x,1)=0$, and $f(x,0)>0$.
They provided conditions under which transition fronts of \eqref{main-eq} exist and also showed that
\eqref{main-eq} may not have transition fronts. Hence the localized spatial inhomogeneity of the media
may prevent
the existence of transition fronts.

We remark that in literature \eqref{main-eq} (resp. \eqref{nonlocal-eq}, \eqref{discrete-eq}) with $f_1(x,u)$ (resp. $f_2(x,u)$, $f_3(j,u)$)
being decreasing in $u$ and negative for $u\gg 1$  and $u\equiv 0$ being an unstable solution is called a Fisher type or KPP type or monostable equation.
The reader is referred to \cite{BeHa}, \cite{MeRoSi}, \cite{NoRoRyZl}, and references therein
for the study of transition solutions of general spatially inhomogeneous Fisher or KPP type equations
and to \cite{HuSh}, \cite{She2}-\cite{She5} for the study of spatial spreading dynamics of
general temporally inhomogeneous Fisher or KPP type equations.

We also remark that it would be interesting to study the spatial spreading dynamics of KPP type equations in inhomogeneous media with
more general limit media, say, equation (1.i) $(i=1,2,3$) with $f_i(x,u)$ being replaced by $f_i(t,x,u)$ satisfying that
$f_i(t,x,u)-f_i^0(t,x,u)\to 0$ as $\|x\|\to \infty$ for some function $f_i^0(t,x,u)$  which is periodic in $t$ and/or $x$. We will consider
such general case elsewhere.

The rest of the paper is organized as follows. In section 2, we introduce
 the standing notions to be used in the paper and the definition of spreading speeds and
 state the main results of the paper (i.e. Theorems \ref{positive-solution-thm}, \ref{spreading-speed-thm}, and \ref{spreading-speed-thm1}).
 In section 3,  we present  some preliminary materials to be used
in later sections.  Section 4 is devoted to the study of  positive stationary solutions of \eqref{main-eq}-\eqref{discrete-eq}.  Theorem \ref{positive-solution-thm}
is proved in this section.
In section 5, we explore the existence of spreading speeds of \eqref{main-eq}-\eqref{discrete-eq}
and prove Theorems \ref{spreading-speed-thm} and \ref{spreading-speed-thm1}.

\section{Standing Notions, Definitions, and Main Results}

In this section, we first  introduce some standing notations and the definition of spreading speeds.
We then state the main results of the paper.

Let $\mathcal{H}_i$ be as in \eqref{habitat-eq}.
Let $p=(p_1,p_2,\cdots,p_N)$ with $p_i>0$ for $i=1,2,\cdots, N$. We define the Banach spaces
$X_{i,p}$ ($i=1,2$) by
\vspace{-0.05in}\begin{equation}
\label{x1-p-space}
X_{1,p}=\{u\in C(\RR^N,\RR)\,|\, u(\cdot+p_{i}{\bf e_{i}})=u(\cdot),\quad i=1,...,N\}
\vspace{-0.05in}\end{equation}
with norm $\|u\|_{X_{1,p}}=\max_{x\in\RR^N}|u(x)|$, and
\vspace{-0.05in}\begin{equation}
\label{x2-p-space}
X_{2,p}=X_{1,p}
\vspace{-0.05in}\end{equation}
(the introduction of $X_{2,p}$ is for the convenience in notation).
If $p_i\in \NN$, we define $X_{3,p}$ by
\vspace{-0.05in}\begin{equation}
\label{x3-p-space}
X_{3,p}=\{u\in C(\ZZ^N,\RR)\,|\, u(\cdot+p_i{\bf e_i})=u(\cdot),\,\,\, i=1,2,\cdots,N\}
\vspace{-0.05in}\end{equation}
with norm $\|u\|_{X_{3,p}}=\max_{j\in\ZZ^N}|u(j)|$.
Let
\vspace{-0.05in}\begin{equation}
\label{x-p-positive}
X_{i,p}^+=\{u\in X_{i,p}\,|\, u(x)\geq 0\,\, \forall x\in\mathcal{H}_i\}
\vspace{-0.05in}\end{equation}
and
\vspace{-0.05in}\begin{equation}
\label{x-p-positive-positive}
X_{i,p}^{++}=\{u\in X_{i,p}\,|\, u(x)> 0\,\, \forall x\in\mathcal{H}_i\}
\vspace{-0.05in}\end{equation}
for  $i=1,2,3$.
We define $X_i$ ($i=1,2,3$) by
\vspace{-0.05in}\begin{equation}
\label{x1-space}
X_1=\{u\in C(\RR^N,\RR)\,|\, u\,\,\, \text{is uniformly continuous and bounded}\}
\vspace{-0.05in}\end{equation}
with norm $\|u\|_{X_1}=\sup_{x\in\RR^N}|u(x)|$,
\vspace{-0.05in}\begin{equation}
\label{x2-space}
X_2=X_1
\vspace{-0.05in}\end{equation}
(again the introduction of $X_2$ is for the convenience in notation),
and
\vspace{-0.05in}\begin{equation}
\label{x3-space}
X_3=\{u\in C(\ZZ^N,\RR)\,|\, u\,\,\, \text{is bounded}\}
\vspace{-0.05in}\end{equation}
with norm $\|u\|_{X_3}=\sup_{j\in\ZZ^N}|u(j)|$.
Let
\vspace{-0.05in}\begin{equation}
\label{x-positive}
X^+_i=\{u\in X_i\,|\, u(x)\geq 0\,\,\forall x\in\mathcal{H}_i\}
\vspace{-0.05in}\end{equation}
and
\vspace{-0.05in}\begin{equation}
\label{x-positive-positive}
X_i^{++}=\{u\in X_i^+\,|\, \inf_{x\in\mathcal{H}_i}u(x)>0\}
\vspace{-0.05in}\end{equation}
for $i=1,2,3$.

If no confusion occurs, we may  write $\|\cdot\|_{X_{i,p}}$ and $\|\cdot\|_{X_i}$ as
$\|\cdot\|$ $(i=1,2,3$).

Assume ${\rm (H1)}$. By general semigroup theory (see \cite{Hen},
\cite{Paz}), for any $u_0\in X_1$ (resp. $u_0\in X_2$, $u_0\in
X_3$), \eqref{main-eq} (resp. \eqref{nonlocal-eq},
\eqref{discrete-eq}) has a unique local solution $u_1(t,\cdot;u_0)$
 (resp. $u_2(t,\cdot;u_0)$,
$u_3(t,\cdot;u_0)$) with $u_1(0,\cdot;u_0)=u_0(\cdot)$ (resp. $u_2(0,\cdot;u_0)=u_0(\cdot)$, $u_3(0,\cdot;u_0)=u_0(\cdot)$). Moreover, if $u_0\in X_i^+$, then
$u_i(t,\cdot;u_0)$ exist and $u_i(t,\cdot;u_0)\in X_i^+$ for all
$t\geq 0$ ($i=1,2,3$) (see Proposition
\ref{basic-global-existence}).

Let
\vspace{-0.05in}\begin{equation}
\label{unit-sphere}
S^{N-1}=\{\xi\in\RR^N\,|\, \|\xi\|=1\}.
\vspace{-0.05in}\end{equation}
For given $\xi\in S^{N-1}$ and $u\in X_i^+$,
we define
\vspace{-0.05in}$$
\liminf_{x\cdot\xi\to -\infty}u(x)=\liminf_{r\to -\infty}\inf_{x\in\mathcal{H}_i,x\cdot\xi\leq r}u(x).
\vspace{-0.05in}$$
For given $u:[0,\infty)\times\mathcal{H}_i\to \RR$ ($1\le i\le 3$) and $c>0$, we define
\vspace{-0.05in}$$
\liminf_{x\cdot\xi\leq ct, t\to\infty}u(t,x)=\liminf_{t\to\infty}\inf_{x\in\mathcal{H}_i,x\cdot\xi\leq ct}u(t,x),
\vspace{-0.05in}$$
$$
\limsup_{x\cdot\xi\geq ct,t\to\infty}u(t,x)=\limsup_{t\to\infty}\sup_{x\in\mathcal{H}_i,x\cdot\xi\geq ct}u(t,x).
$$
The notions $\disp\limsup_{|x\cdot\xi|\leq ct,t\to\infty}u(t,x)$, $\disp\limsup_{|x\cdot\xi|\ge ct,t\to\infty}u(t,x)$,
$\disp\limsup_{\|x\|\leq ct,t\to\infty}u(t,x)$, and $\disp\limsup_{\|x\|\geq ct,t\to\infty}u(t,x)$ are defined similarly.
We define $X_i^+(\xi)$ ($i=1,2,3$) by
\vspace{-0.05in}\begin{equation}
\label{x-xi-space}
X_i^+(\xi)=\{u\in X_i^+\,|\, \liminf_{x\cdot\xi\to -\infty}u(x)>0,\quad u(x)=0\,\, {\rm for}\,\,  x\cdot\xi\gg 1\}.
\vspace{-0.05in}\end{equation}

\begin{definition} [Spatial spreading speed]
\label{spreading-def}
For given $\xi\in S^{N-1}$ and given $i\in\NN$ $(1\leq i\leq 3)$, a real number $c_i^*(\xi)$
 is called the {\rm spatial spreading speed} of (1.i) in the direction of $\xi$ if for any $u_0\in X_i^+(\xi)$,
\vspace{-0.05in}$$
\liminf_{x\cdot\xi\leq ct, t\to \infty}u_i(t,x;u_0)>0\quad \forall c<c_i^*(\xi)
\vspace{-0.05in}$$
and
\vspace{-0.05in}$$
\limsup_{x\cdot\xi\geq ct, t\to\infty}u_i(t,x;u_0)=0\quad \forall c>c_i^*(\xi).
\vspace{-0.05in}$$
\end{definition}

The main results of this paper are stated in the following three theorems.

\begin{theorem} [Positive stationary solutions]
\label{positive-solution-thm}
Assume ${\rm (H1)}$ and ${\rm (H2)}$.
\begin{itemize}
\vspace{-0.05in}\item[(1)] $($Existence$)$
 Equation \eqref{main-eq} $($resp. \eqref{nonlocal-eq},
 \eqref{discrete-eq}$)$
  has a unique stationary solution $u=u_1^*(\cdot)\in X_1^{++}$
  $($resp. $u=u_2^*(\cdot)\in X_2^{++}$, $u=u_3^*(\cdot)\in X_3^{++})$. Moreover,
$$\lim_{r\to\infty}\sup_{x\in\mathcal{H}_i,\|x\|\geq r}|u_i^*(x)-u_i^0|=0,$$
where $u_i^0>0$ is such that $f_i^0(u_i^0)=0$ and $i=1,2,3$.

\vspace{-0.05in}\item[(2)] $($Stability$)$
For any $u_0\in X_i^{++}$, $\lim_{t\to\infty}\|u_i(t,\cdot;u_0)-u_i^*(\cdot)\|_{X_i}=0.$

\vspace{-0.05in}\item[(3)] $($Stability$)$ For any $u_0\in X_i^+\setminus\{0\}$,
$\lim_{t\to\infty} u_i(t,x;u_0)=u_i^*(x)$
uniformly in $x$ on bounded sets.
\end{itemize}
\end{theorem}

\begin{theorem} [Existence and characterization of spreading speeds]
\label{spreading-speed-thm} Assume ${\rm (H1)}$ and ${\rm (H2)}$. Then for any given
$\xi\in S^{N-1}$, \eqref{main-eq} $($resp. \eqref{nonlocal-eq},
\eqref{discrete-eq}$)$ has a spreading speed $c_1^*(\xi)$
$($resp. $c_2^*(\xi)$, $c_3^*(\xi))$ in the direction of $\xi$. Moreover,
for any $u_0\in X_i^+(\xi)$,
\vspace{-0.05in}\begin{equation}
\label{spreading-eq0}
\liminf_{x\cdot\xi\leq ct, t\to \infty}|u_i(t,x;u_0)-u_i^*(x)|=0\quad \forall c<c_i^*(\xi),
\vspace{-0.05in}\end{equation}
and
\vspace{-0.05in} $$c_i^*(\xi)=c_i^0(\xi)\quad \text{for}\quad  i=1,2,3,
\vspace{-0.05in} $$
 where
\vspace{-0.05in}\begin{equation}
\label{c-1-0-eq}
c_1^0(\xi)=\inf_{\mu>0}\frac{f_1^0(0)+\mu^2}{\mu}= 2\sqrt {f_1^0(0)},
\vspace{-0.05in}\end{equation}
\vspace{-0.05in}\begin{equation}
\label{c-2-0-eq}
c_2^0(\xi)=\inf_{\mu>0}\frac{\int_{\RR^N}e^{-\mu z\cdot\xi}\kappa(z)dz-1+f_2^0(0)}{\mu},
\vspace{-0.05in}\end{equation}
and
\vspace{-0.05in}\begin{equation}
\label{c-3-0-eq}
c_3^0(\xi)= \inf_{\mu>0}\frac{\sum_{k\in K}a_k(e^{-\mu k\cdot\xi}-1)+f_3^0(0)}{\mu}
\vspace{-0.05in}\end{equation}
are the spatial spreading speeds of \eqref{limit-main-eq}, \eqref{limit-nonlocal-eq}, and \eqref{limit-discrete-eq} in the direction of $\xi$, respectively.
\end{theorem}

\begin{theorem}[Spreading features of spreading speeds]
\label{spreading-speed-thm1}
Assume ${\rm (H1)}$ and ${\rm (H2)}$ and $1\leq i\leq 3$.  Then for any given $\xi\in S^{N-1}$, the following hold.
\begin{itemize}
\vspace{-0.05in}\item[(1)] For each $ u_0\in X_i^+$ satisfying that  $u_0(x)=0$ for $x\in\mathcal{H}_i$ with $|x\cdot\xi|\gg 1$,
\vspace{-0.05in}$$
\limsup_{|x\cdot\xi|\geq ct, t\to\infty}u_i(t,x;u_0)=0 \quad \forall c>\max\{c^*_i(\xi),c^*_i(-\xi)\}.
\vspace{-0.05in}$$

\vspace{-0.05in}\item[(2)]  For each $\sigma>0$,  $r>0$, and  $u_0\in X^+_i$ satisfying that $u_0(x)\geq\sigma$ for $x\in\mathcal{H}_i$ with $|x\cdot \xi|\leq r$,
\vspace{-0.05in}$$
\limsup_{|x\cdot \xi|\leq ct,t\to\infty}|u_i(t,x;u_0)-u_i^*(x)|=0\quad \forall 0<c<\min\{c^*_i(\xi),c^*_i(-\xi)\}.
\vspace{-0.05in}$$

\vspace{-0.05in}\item[(3)] For each $u_0\in X_i^+$ satisfying that $u_0(x)=0$ for $x\in\mathcal{H}_i$ with $\|x\|\gg 1$,
\vspace{-0.05in}$$
\limsup_{\|x\|\geq ct,t\to\infty}u_i(t,x;u_0)=0\quad \forall c>\sup_{\xi\in S^{N-1}}c_i^*(\xi).
\vspace{-0.05in}$$

\vspace{-0.05in}\item[(4)] For each $\sigma>0$, $r>0$, and  $u_0\in X_i^+$ satisfying that $u_0(x)\geq \sigma$ for $\|x\|\le r$,
\vspace{-0.05in}$$
\limsup_{\|x\|\leq ct,t\to\infty}|u_i(t,x;u_0)-u_i^*(x)|=0\quad \forall 0<c<\inf_{\xi\in S^{N-1}}c_i^*(\xi).
\vspace{-0.05in}$$
\end{itemize}
\end{theorem}

To indicate the dependence of $u_i^*(\cdot)$ and $c_i^*(\xi)$ on $f_i$, we may sometime write
$u_i^*(\cdot)$ and $c_i^*(\xi)$ as $u_i^*(\cdot;f_i(\cdot,\cdot))$ and $c_i^*(\xi;f_i(\cdot,\cdot))$, respectively.

\section{Preliminary}

In this section, we present some preliminary materials to be used in later sections, including some basic properties of solutions of \eqref{main-eq}-\eqref{discrete-eq};
principal
eigenvalue theories for spatially periodic dispersal operators with random, nonlocal, and discrete dispersals; and
spatial spreading dynamics of KPP equations in spatially periodic media.

\subsection{Basic properties of KPP equations}

In this subsection, we present some basic properties of  solutions of \eqref{main-eq}-\eqref{discrete-eq}, including comparison principle,
global existence, convergence in open compact topology, and decreasing of the so called part metric along the solutions.
Throughout this subsection, we assume ${\rm (H1)}$.

Let $X_1$, $X_2$, and $X_3$ be as in \eqref{x1-space},
\eqref{x2-space}, and \eqref{x3-space}, respectively.  For given
$u_0\in X_1$ $($resp. $u_0\in X_2$, $u_0\in X_3$), let
$u_1(t,\cdot;u_0)$ $($resp. $u_2(t,\cdot;u_0)$, $u_3(t,\cdot;u_0)$)
be the (local) solution of \eqref{main-eq} $($resp.
\eqref{nonlocal-eq}, \eqref{discrete-eq}$)$  with
$u_1(0,\cdot;u_0)=u_0(\cdot)$ $($resp.
$u_2(0,\cdot;u_0)=u_0(\cdot)$, $u_3(0,\cdot;u_0)=u_0(\cdot)$).

Let $X_i^+$ and $X_i^{++}$ ($i=1,2,3)$
be as in \eqref{x-positive} and \eqref{x-positive-positive}. For given $1\leq i\leq 3$ and  $u,v\in X_i$, we define
\vspace{-0.05in}\begin{equation}
\label{order-eq} u\leq v\,\, (u\geq v)\quad {\rm if}\,\, v-u\in
X_i^+\,\, (u-v\in X_i^+)
\vspace{-0.05in}\end{equation}
and
\vspace{-0.05in}\begin{equation}
\label{order-order-eq}
u\ll v\,\, (u\gg v)\quad {\rm if}\,\, v-u\in X_i^{++}\,\, (u-v\in X_i^{++}).
\vspace{-0.05in}\end{equation}

For given continuous and bounded function $u:[0,T)\times\RR^N\to\RR$, it is called a {\it super-solution} ({\it sub-solution}) of \eqref{main-eq} on $[0,T)$ if
\vspace{-0.05in}$$
u_t(t,x)\geq (\leq) \Delta u(t,x)+u(t,x)f_1(x,u(t,x))\quad\forall  (t,x)\in (0,T)\times \RR^N.
\vspace{-0.05in}$$
Super-solutions (sub-solutions) of \eqref{nonlocal-eq} and \eqref{discrete-eq} are defined similarly.

\begin{proposition}[Comparison principle]
\label{basic-comparison} Assume ${\rm (H1)}$.
\begin{itemize}
\vspace{-0.05in}\item[(1)] Suppose that $u^1(t,x)$ and $u^2(t,x)$ are sub- and super-solutions of \eqref{main-eq} $($resp. \eqref{nonlocal-eq}, \eqref{discrete-eq}$)$
on $[0,T)$ with $u^1(0,\cdot)\leq u^2(0,\cdot)$. Then
$u^1(t,\cdot)\leq u^2(t,\cdot)$ for $t\in (0,T)$. Moreover, if
$u^1(0,\cdot)\not = u^2(0,\cdot)$, then $u^1(t,x)<u^2(t,x)$ for
$x\in\mathcal{H}_1$ $($resp. $x\in\mathcal{H}_2$,
$x\in\mathcal{H}_3)$ and $t\in(0,T)$.

\vspace{-0.05in}\item[(2)] If $u_{01},u_{02}\in X_i$ and $u_{01}\leq u_{02}$ $(1\leq i\leq 3)$, then $u_i(t,\cdot;u_{01})\leq u_i(t,\cdot;u_{02})$ for $t>0$ at which both
$u_i(t,\cdot;u_{01})$ and $u_i(t,\cdot;u_{02})$  exist.

\vspace{-0.05in}\item[(3)]  If $u_{01},u_{02}\in X_i$ and $u_{01}\leq u_{02}$, $u_{01}\not =u_{02}$ $(1\leq i\leq 3)$, then $u_i(t,x;u_{01})< u_i(t,x;u_{02})$ for all $x\in\mathcal{H}_i$
 and $t>0$ at which both
$u_i(t,\cdot;u_{01})$ and $u_i(t,\cdot;u_{02})$  exist.

\vspace{-0.05in}\item[(4)] If $u_{01},u_{02}\in X_i$ and $u_{01}\ll u_{02}$ $(1\leq i\leq 3)$, then  $u_i(t,\cdot;u_{01})\ll u_i(t,\cdot;u_{02})$ for $t>0$ at which both
$u_i(t,\cdot;u_{01})$ and $u_i(t,\cdot;u_{02})$  exist.
\end{itemize}
\end{proposition}

\begin{proof}
(1) The case $i=1$ follows from comparison principle for parabolic equations.
The case $i=2$ follows from \cite[Propositions 2.1 and 2.2]{ShZh1}.
The case $i=3$  follows from comparison principle for lattice differential equations
(see the arguments in \cite[Lemma 1]{ChGuWu}).

(2) and (3) follow from (1).

(4) We provide a proof for the case $i=2$. Other cases can be proved similarly.
Take any $T>0$ such that both $u_2(t,\cdot;u_{01})$ and $u_2(t,\cdot;u_{02})$ exist on
$[0,T]$. It suffices to prove that $u_2(t,\cdot;u_{02})\gg u_2(t,\cdot;u_{01})$
for $t\in [0,T]$. To this end, let $w(t,x)=u_2(t,x;u_{02})-u_2(t,x;u_{01})$. Then $w(t,x)$ satisfies the following equation,
\vspace{-0.05in}\begin{equation*}
w_t(t,x)=\int_{\RR^N}\kappa(y-x)w(t,y)dy-w(t,x)+a(t,x)w(t,x),
\vspace{-0.05in}\end{equation*}
where
\vspace{-0.05in}\begin{align*}
a(t,x)=&f_2(x,u_2(t,x;u_{02}))\\
&+u_2(t,x;u_{01})\int_0^1 \p_u f_2(x,s u_2(t,x;u_{02})+(1-s)u_2(t,x;u_{01}))ds.
\vspace{-0.05in}\end{align*}
Let $M>0$ be such that
$M\geq \sup_{x\in\RR^N,t\in[0,T]}(1-a(t,x))$
and $\tilde w(t,x)=e^{Mt}w(t,x)$.
Then $\tilde w(t,x)$ satisfies
\vspace{-0.05in}$$
\tilde w_t(t,x)=\int_{\RR^N}\kappa(y-x)\tilde w(t,y)dy+[M-1+a(t,x)]\tilde w(t,x).
\vspace{-0.05in}$$
Let $\mathcal{K}:X_2\to X_2$ be defined by
\vspace{-0.05in}\begin{equation}
\label{k-op}
(\mathcal{K}u)(x)=\int_{\RR^N}\kappa(y-x)u(y)dy\quad {\rm for}\quad u\in X_2.
\vspace{-0.05in}\end{equation}
Then $\mathcal{K}$ generates an analytic semigroup on $X_2$ and
\vspace{-0.05in}$$
\tilde w(t,\cdot)=e^{\mathcal{K}t}(u_{02}-u_{01})+\int_0^t e^{\mathcal{K}(t-\tau)}(M-1+a(\tau,\cdot))\tilde w(\tau,\cdot)d\tau.
\vspace{-0.05in}$$
Observe that $e^{\mathcal{K} t}u_0\ge 0$ for any $u_0\in X_2^+$ and $t\ge 0$ and
$e^{\mathcal{K}t}u_0\gg 0$ for any $u_0\in X_2^{++}$ and $t\ge 0$. Observe also that
$u_{02}-u_{01}\in X_2^{++}$. By (2), $\tilde w(\tau,\cdot)\geq 0$ and hence $(M-1+a(\tau,\cdot))\tilde w(\tau,\cdot)\ge 0$
for $\tau\in[0,T]$.  It then follows that
$\tilde w(t,\cdot)\gg 0$ and then $w(t,\cdot)\gg 0$ (i.e. $u_2(t,\cdot;u_{02})\gg u_2(t,\cdot;u_{01})$) for $t\in [0,T]$.
\end{proof}

\begin{proposition}[Global existence]
\label{basic-global-existence}
Assume ${\rm (H1)}$.
For any given $1\leq i\leq 3$ and $u_0\in X_i^+$, $u_i(t,\cdot;u_0)$ exists for all $t\geq 0$.
\end{proposition}

\begin{proof}
Let $1\leq i\leq 3$ and $u_0\in X_i^+$ be given. There is $M\gg 1$ such that
$0\leq u_0(x)\leq M$ and $f_i(x,M)<0$ for all $x\in\mathcal{H}_i$. Then by Proposition \ref{basic-comparison},
\vspace{-0.08in}$$
0\leq u_i(t,\cdot;u_0)\leq M
\vspace{-0.08in}$$
for any $t>0$ at which $u_i(t,\cdot;u_0)$ exists. It is then not difficult to prove that
for any $T>0$ such that $u_i(t,\cdot;u_0)$ exists on $(0,T)$, $\lim_{t\to T}u_i(t,\cdot;u_0)$ exists
in $X_i$.
 This implies that $u_i(t,\cdot;u_0)$ exists and
$u_i(t,\cdot;u_0)\geq 0$ for all $t\geq 0$.
\end{proof}

For given $u,v\in X_i^{++}$, define
\vspace{-0.05in}$$
\rho_i(u,v)=\inf\{\ln \alpha\,|\, \frac{1}{\alpha}u\leq v\leq \alpha u,\,\, \alpha\geq 1\}.
\vspace{-0.05in}$$
Observe that $\rho_i(u,v)$ is well defined and there is $\alpha\geq 1$ such that $\rho_i(u,v)=\ln\alpha$.
Moreover, $\rho_i(u,v)=\rho_i(v,u)$ and $\rho_i(u,v)=0$ iff $u\equiv v$. In literature, $\rho_i(u,v)$ is called the {\it part metric}
between $u$ and $v$.

\begin{proposition}[Decreasing of part metric]
\label{basic-part-metric}
For given $1\leq i\leq 3$ and $u_0,v_0\in X^{++}_i$ with $u_0\not =v_0$,
$\rho_i(u_i(t,\cdot;u_0),u_i(t,\cdot;v_0))$ is non-increasing in $t\in (0,\infty)$.
\end{proposition}

\begin{proof}
We give a proof for the case $i=1$. Other cases can be proved similarly.

First, note that there is $\alpha^*>1$ such that  $\rho_1(u_0,v_0)=\ln \alpha^{*}$
and $\frac{1}{\alpha^{*}}u_0\leq v_0\leq \alpha^{*} u_0$. By Proposition \ref{basic-comparison},
\vspace{-0.05in}$$u_1(t, \cdot;v_0)\le u_1(t, \cdot; \alpha^{*}u_0)\quad {\rm for}\quad t>0.
\vspace{-0.05in}$$
Let $v(t, x)=\alpha^{*}u_1(t, x; u_0)$.
Then
\vspace{-0.05in}\begin{align*}
v_{t}(t,x)&=\Delta v(t,x)+v(t,x)f_1(x, u_1(t,x;u_0))\\
& =\Delta
v(t,x)+v(t,x)f_1(x,v(t,x))+v(t,x)f_1(x,u_1(t,x;u_0))-v(t,x)f_1(x,v(t,x))\\
&>\Delta v (t,x)+ v(t,x)f_1(x,v(t,x)).
\vspace{-0.05in}\end{align*}
This together with Proposition \ref{basic-comparison} implies that
\vspace{-0.05in}$$
u_1(t,\cdot;\alpha^* u_0)\leq \alpha^* u_1(t,\cdot;u_0)\quad {\rm for}\quad t>0
\vspace{-0.05in}$$
and then
\vspace{-0.05in}$$
u_1(t,\cdot;v_0)\le \alpha^* u_1(t,\cdot;u_0)\quad {\rm for}\quad t>0.
\vspace{-0.05in}$$
Similarly, it can be proved that
\vspace{-0.05in}$$
\frac{1}{\alpha^*}u_1(t,\cdot;u_0)\le u_1(t,\cdot;v_0)\quad {\rm for}\quad t>0.
\vspace{-0.05in}$$
It then follows that
\vspace{-0.05in}$$
\rho_1(u_1(t,\cdot;u_0),u_1(t,\cdot;v_0))\le \rho_1(u_0,v_0)\quad \forall t>0
\vspace{-0.05in}$$
and hence
\vspace{-0.05in}$$
\rho_1(u_1(t_2,\cdot;u_0),u_1(t_2,\cdot;v_0))\le\rho_1(u_{1}(t_1,\cdot;u_0),u_1(t_1,\cdot;v_0))\quad\forall
0\le t_1<t_2.
\vspace{-0.05in}$$
\end{proof}

To indicate the dependence of solutions of \eqref{main-eq}-\eqref{discrete-eq} on the nonlinearity,
we may write $u_i(t,\cdot;u_0)$ as $u_i(t,\cdot;u_0,f_i(\cdot,\cdot))$. Observe that for any $z_n\in\mathcal{H}_i$,
if $\{z_n\}$ is a bounded sequence, then there are $z^*\in\mathcal{H}_i$ and $\{z_{n_k}\}\subset\{z_n\}$ such that
$z_{n_k}\to z^*$ and $f_i(x+z_{n_k},u)\to f_i(x+z^*,u)$ uniformly in $(x,u)$ on bounded sets.
If $\{z_n\}$ is an unbounded sequence, then there is $z_{n_k}$ such that
$f_i(x+z_{n_k},u)\to f_i^0(u)$ uniformly in $(x,u)$ on bounded sets.

\begin{proposition}[Convergence on compact subsets]
\label{basic-convergence}
Given $1\leq i\leq 3$, suppose that $u_{0n},u_0\in X_i^+$ ($n=1,2,\cdots$), $\{\|u_{0n}\|\}$ is bounded, and $u_{0n}(x)\to u_0(x)$ as $n\to\infty$
 uniformly in $x$ on bounded sets.
\begin{itemize}
\vspace{-0.05in}\item[(1)] If $z_n,z^*\in\mathcal{H}_i$ $(n=1,2,\cdots)$ are such that  $f_i(x+z_n,u)\to f_i(x+z^*,u)$ as $n\to\infty$  uniformly
in $(x,u)$ on bounded sets, then for each $t>0$,
 $u_i(t,x;u_{0n},f_i(\cdot+z_n,\cdot))\to u_i(t,x;u_0,f_i(\cdot+z^*,\cdot))$ as $n\to\infty$ uniformly in $x$ on bounded sets.

\vspace{-0.05in}\item[(2)] If $z_n\in\mathcal{H}_i$ $(n=1,2,\cdots)$ are such that $f_i(x+z_n,u)\to f_i^0(u)$ as $n\to\infty$   uniformly
in $(x,u)$ on bounded sets, then for each $t>0$,
 $u_i(t,x;u_{0n},f_i(\cdot+z_n,\cdot))\to u_i(t,x;u_0,f_i^0(\cdot))$ as $n\to\infty$ uniformly in $x$ on bounded sets.
\end{itemize}
\end{proposition}

\begin{proof}
We prove (1) with $i=2$. All other cases can be proved similarly.

 Let $v^n(t,x)=u_2(t,x;u_{0n},f_2(\cdot+z_n,\cdot))-u_2(t,x;u_0,f_2(\cdot+z^*,\cdot))$.
Then $v^n(t,x)$ satisfies
\vspace{-0.05in}\begin{equation*}
v^n_t(t,x)=\int_{\RR^N}\kappa(y-x)v^n(t,y)dy-v^n(t,x)+a_n(t,x)v^n(t,x)+b_n(t,x),
\vspace{-0.05in}\end{equation*}
where
\vspace{-0.05in}\begin{align*}
a_n(t,x)=&f_2(x+z_n,u_2(t,x;u_{0n},f_2(\cdot+z_n,\cdot)))+u_2(t,x;u_0,f_2(\cdot+z^*,\cdot))\\
&\cdot \int_0^1 \p_u f_2(x+z_n,s u_2(t,x;u_{0n},f_2(\cdot+z_n,\cdot))+
(1-s)u_2(t,x;u_0,f_2(\cdot+z^*,\cdot)))ds
\vspace{-0.05in}\end{align*}
and
\vspace{-0.05in}\begin{align*}
b_n(t,x)= &u_2(t,x;u_0,f_2(\cdot+z^*,\cdot))\\
&\cdot \big(f_2(x+z_n, u_2(t,x;u_0,f_2(\cdot+z^*,\cdot)))-f_2(x+z^*,
 u_2(t,x;u_0,f_2(\cdot+z^*,\cdot)))\big).
\vspace{-0.05in}\end{align*}
Observe that $\{a_n(t,x)\}$ is uniformly bounded and continuous  in $t$ and $x$ and
$b_n(t,x)\to 0$ as $n\to\infty$ uniformly in $t\in[0,\infty)$ and $x$ on bounded sets.

Take a $\rho>0$. Let
\vspace{-0.05in}$$
X_2(\rho)=\{u\in C(\RR^N,\RR)\,|\, u(\cdot) e^{-\rho \|\cdot\|}\in X_2\}
\vspace{-0.05in}$$
with norm $\|u\|_\rho=\|u(\cdot)e^{-\rho \|\cdot\|}\|$.
Note that $\mathcal{K}:X_2(\rho)\to X_2(\rho)$ also generates an analytic  semigroup,
where $\mathcal{K}$ is as in \eqref{k-op},
and there are $M>0$ and $\omega>0$ such that
\vspace{-0.05in}$$
\|e^{(\mathcal{K}-\mathcal{I})t}\|_{X_2(\rho)}\leq M e^{\omega t}\quad \forall t\geq 0,
\vspace{-0.05in}$$
where $\mathcal{I}$ is the identity map on $X_2(\rho)$.
Hence
\vspace{-0.05in}\begin{align*}
v^n(t,\cdot)=&e^{(\mathcal{K}-\mathcal{I})t}v^n(0,\cdot)+\int_0^t e^{(\mathcal{K}-\mathcal{I})(t-\tau)}a_n(\tau,\cdot)v^n(\tau,\cdot)
d\tau\\
&+\int_0^t e^{(\mathcal{K}-\mathcal{I})(t-\tau)}b_n(\tau,\cdot)d\tau
\vspace{-0.05in}\end{align*}
and
then
\vspace{-0.05in}\begin{align*}
\|v^n(t,\cdot)\|_{X_2(\rho)}&\leq M e^{\omega t}\|v^n(0,\cdot)\|_{X_2(\rho)}+
M\sup_{\tau\in[0,t],x\in\RR^N} |a_n(\tau,x)|\int_0^ t e^{\omega(t-\tau)}\|v^n(\tau,\cdot)\|_{X_2(\rho)}d\tau\\
&\quad +M\int_0^ t e^{\omega(t-\tau)}\|b_n(\tau,\cdot)\|_{X_2(\rho)}d\tau\\
&\leq M e^{\omega t}\|v^n(0,\cdot)\|_{X_2(\rho)}+
M\sup_{\tau\in[0,t],x\in\RR^N} |a_n(\tau,x)|\int_0^ t e^{\omega(t-\tau)}\|v^n(\tau,\cdot)\|_{X_2(\rho)}d\tau\\
&\quad + \frac{M}{\omega}\sup_{\tau\in
[0,t]}\|b_n(\tau,\cdot)\|_{X_2(\rho)} e^{\omega t}.
\vspace{-0.05in}\end{align*}
By Gronwall's inequality,
\vspace{-0.05in}$$
\|v^n(t,\cdot)\|_{X_2(\rho)}\leq e^{(\omega+M\sup_{\tau\in[0,t],x\in\RR^N} |a_n(\tau,x)|)t}\Big(M\|v^n(0,\cdot)\|_{X_2(\rho)}
+ \frac{M}{\omega}\sup_{\tau\in [0,t]}\|b_n(\tau,\cdot)\|_{X_2(\rho)}\Big).
\vspace{-0.05in}$$
Note that $\|v^n(0,\cdot)\|_{X_2(\rho)}\to 0$ and $\sup_{\tau\in
[0,t]}\|b_n(\tau,\cdot)\|_{X_2(\rho)}\to 0$ as $n\to\infty$. It then
follows that
\vspace{-0.05in}$$
\|v^n(t,\cdot)\|_{X_2(\rho)}\to 0\quad {\rm as}\quad n\to\infty
\vspace{-0.05in}$$ and then
\vspace{-0.05in}$$
u_2(t,x;u_{0n},f_2(\cdot+z_n,\cdot))\to u_2(t,x;u_0,f_2(\cdot+z^*,\cdot))\quad {\rm as}\quad
n\to\infty
\vspace{-0.05in}$$
uniformly in $x$ on bounded sets.
\end{proof}

\subsection{Principal eigenvalues of spatially periodic dispersal operators}

In this subsection, we present some principal
eigenvalue theories for spatially periodic dispersal operators with random, nonlocal, and discrete dispersals.

Let $p=(p_1,p_2,\dots,p_N)$ with $p_i>0$ for $i=1,2,\cdots,N$ and $X_{i,p}$ be as in \eqref{x1-p-space}-\eqref{x3-p-space}. When $X_{3,p}$ is considered, it is assumed that
$p_i\in\NN$.
 We will denote $\mathcal{I}$ as an identity map on
the  Banach space under consideration.
For given $\xi\in S^{N-1}$, $\mu\in\RR$,   $a_i\in X_{i,p}$ ($i=1,2,3$), consider the following
eigenvalue problems,
\begin{equation}
\label{random-eigenvalue-eq}
\begin{cases}\Delta u(x)-2\mu \xi\cdot\nabla u(x)+(a_1(x)+\mu^2)u(x)=\lambda u(x),\quad x\in\RR^N\cr
u(x+p_i{\bf e_i})=u(x),\quad x\in\RR^N,
\end{cases}
\end{equation}
\begin{equation}
\label{nonlocal-eigenvalue-eq}
\begin{cases}
\int_{\RR^N} e^{-\mu (y-x)\cdot\xi}\kappa(y-x)u(y)dy-u(x)+a_2(x)u(x)=\lambda u(x),\quad x\in\RR^N\cr
u(x+p_i{\bf e_i})=u(x),\quad x\in\RR^N
\end{cases}
\end{equation}
and
\begin{equation}
\label{discrete-eigenvalue-eq}
\begin{cases}
\sum_{k\in K}a_k(e^{-\mu k\cdot\xi} u(j+k)-u(j))+a_3(j)u(j)=\lambda u(j),\quad j\in\ZZ^N\cr
 u(j+p_i{\bf e_i})=u(j),\quad j\in\ZZ^N.
\end{cases}
\end{equation}
Observe that when $\mu=0$, \eqref{random-eigenvalue-eq}, \eqref{nonlocal-eigenvalue-eq}, and
\eqref{discrete-eigenvalue-eq} are independent of $\xi$. Observe also that if $u(t,x)=e^{-\mu (x\cdot\xi-\frac{\lambda}{\mu}t)}\phi(x)$ is a solution of
\vspace{-0.05in}\begin{equation}
\label{random-linear-eq}
u_t(t,x)=\Delta u(t,x)+a_1(x)u(t,x),\quad x\in\RR^N
\vspace{-0.05in}\end{equation}
with $\phi(\cdot)\in X_{1,p}\setminus\{0\}$, or a solution of
\vspace{-0.05in}\begin{equation}
\label{nonlocal-linear-eq}
u_t(t,x)=\int_{\RR^N}k(y-x)u(t,y)dy-u(t,x)+a_2(x)u(t,x),\quad x\in\RR^N
\vspace{-0.05in}\end{equation}
with $\phi(\cdot)\in X_{2,p}\setminus\{0\}$, or a solution
of
\vspace{-0.05in}\begin{equation}
\label{discrete-linear-eq}
u_t(t,j)=\sum_{k\in K}a_k(u(t,x+j)-u(t,j))+a_3(j)u(t,j),\quad j\in\ZZ^N
\vspace{-0.05in}\end{equation}
with $\phi(\cdot)\in X_{3,p}\setminus\{0\}$, then $\lambda$ is an
eigenvalue of \eqref{random-eigenvalue-eq} or
\eqref{nonlocal-eigenvalue-eq} or \eqref{discrete-eigenvalue-eq}
with  $\phi(\cdot)$ being a corresponding eigenfunction. If
$a_1(x)=f_1(x,0)$ $($resp. $a_2(x)=f_2(x,0)$, $a_3(j)=f_3(j,0)$),
then \eqref{random-linear-eq} $($resp. \eqref{nonlocal-linear-eq},
\eqref{discrete-linear-eq}) is the linearized equation of
\eqref{main-eq} $($resp. \eqref{nonlocal-eq}, \eqref{discrete-eq})
at $u=0$.

Define $\mathcal{O}_{i,\mu,\xi}:\mathcal{D}(\mathcal{O}_{i,\mu,\xi})\subset X_{i,p}\to X_{i,p}$ ($i=1,2,3$) by
\vspace{-0.05in}\begin{equation}
\label{o1-operator} (\mathcal{O}_{1,\mu,\xi}u)(x)=\Delta
u(x)-2\mu\xi\cdot\nabla u(x)+(a_1(x)+\mu^2)u(x)\quad \forall\, u\in
\mathcal{D}(\mathcal{O}_{1,\mu,\xi})\subset X_{1,p},
\vspace{-0.05in}\end{equation}
\vspace{-0.05in}\begin{equation}
\label{o2-operator}
(\mathcal{O}_{2,\mu,\xi}u)(x)=\int_{\RR^N}e^{-\mu (y-x)\cdot\xi}\kappa(y-x)u(y)dy-u(x)+a_2(x)u(x)\quad \forall\, u\in \mathcal{D}(\mathcal{O}_{2,\mu,\xi})= X_{2,p}
\vspace{-0.05in}\end{equation}
and
\vspace{-0.05in}\begin{equation}
\label{d-operator}
(\mathcal{O}_{3,\mu,\xi}u)(j)=\sum_{k\in K}a_k( e^{-\mu k\cdot\xi}u(j+k)-u(j))+a_3(j)u(j)\quad \forall\, u\in \mathcal{D}(\mathcal{O}_{3,\mu,\xi})=X_{3,p}.
\vspace{-0.05in}\end{equation}
Let $\sigma(\mathcal{O}_{i,\mu,\xi})$ be the spectrum of $\mathcal{O}_{i,\mu,\xi}$ ($i=1,2,3$).

\begin{definition}
\label{principal-def}
Let $1\leq i\leq 3$, $\mu\in\RR$, and $\xi\in S^{N-1}$ be given.
A real number $\lambda_i(\mu,\xi,a_i)\in\RR$
is called  the
  {\rm principal eigenvalue} of
$\mathcal{O}_{i,\mu,\xi}$
 if
it is an isolated algebraic simple eigenvalue of $\mathcal{O}_{i,\mu,\xi}$ with a positive eigenfunction and for any
$\lambda\in \sigma(\mathcal{O}_{i,\mu,\xi})\setminus\{\lambda_i(\mu,\xi,a_i)\}$,
 ${\rm Re}\lambda<\lambda_i(\mu,\xi,a_i)$.
\end{definition}

For given $1\leq i\leq 3$, $\mu\in\RR$, and $\xi\in S^{N-1}$, let
\vspace{-0.05in}\begin{equation}
\label{principal-spectrum-point}
\lambda^0_i(\mu,\xi,a_i)=\sup\{{\rm Re}\mu\,|\, \mu\in\sigma(\mathcal{O}_{i,\mu,\xi})\}.
\vspace{-0.05in}\end{equation}
Observe that for any $\mu\in\RR$ and $\xi\in S^{N-1}$, $\mathcal{O}_{i,\mu,\xi}$ generates an analytic
semigroup $\{T_i(t)\}_{t\geq 0}$ in $X_{i,p}$
and moreover, $T_i(t)$ is strongly positive (that is, $T_i(t)u_0\geq 0$ for any $t\geq 0$ and $u_0\in X_{i,p}^+$
and $T_i(t)u_0\gg 0$ for any $t>0$ and $u_0\in X_{i,p}^{+}\setminus\{0\}$).
 Then by \cite[Proposition 4.1.1]{Nie}, $r(T_i(t))\in \sigma(T_i(t))$ for any $t>0$, where $r(T_i(t))$ is the spectral
   radius of $T_i(t)$.
   Hence by the spectral mapping theorem (see \cite[Theorem 2.7]{ChLa}), $\lambda_i^0(\mu,\xi,a_i)\in \sigma(\mathcal{O}_{i,\mu,\xi})$ for $i=1,2,3$.
 Observe also that
  $\lambda_i^0(0,\xi,a_i)$ ($i=1,2,3$) are independent of $\xi\in S^{N-1}$. We may then put
$$
\lambda_i^0(a_i)=\lambda_i^0(0,\xi,a_i),\quad i=1,2,3.
$$

It is well known that
the principal eigenvalue $\lambda_1(\mu,\xi,a_1)$ and $\lambda_3(\mu,\xi,a_3)$ of
$\mathcal{O}_{1,\mu,\xi}$ and $\mathcal{O}_{3,\mu,\xi}$ exist for all $\mu\in\RR$ and
$\xi\in S^{N-1}$ and
\vspace{-0.05in}$$
\lambda_i(\mu,\xi,a_i)=\lambda_i^0(\mu,\xi,a_i),\quad i=1,3.
\vspace{-0.05in}$$
The principal eigenvalue of $\mathcal{O}_{2,\mu,\xi}$ may not exist (see an example
in \cite{ShZh1}). If the principal eigenvalue $\lambda_2(\mu,\xi,a_2)$ exists, then
$$
\lambda_2(\mu,\xi,a_2)=\lambda_2^0(\mu,\xi,a_2).
$$
Regarding the existence of principal eigenvalue of $\mathcal{O}_{2,\mu,\xi}$, the following proposition
is proved in \cite{ShZh1}, \cite{ShZh2}.

\begin{proposition}[Existence of principal eigenvalue]
\label{existence-principal-eigenvalue}
\begin{itemize}
\item[(1)] If $a_2\in C^N(\RR^N,\RR)\cap X_{2,p}$ and the partial derivatives of $a_2(x)$ up to order
$N-1$ are zero at some $x_0$ satisfying that $a_2(x_0)=\max_{x\in\RR^N}a_2(x)$, then
the principal eigenvalue $\lambda_2(\mu,\xi,a_2)$ of $\mathcal{O}_{2,\mu,\xi}$ exists for
all $\mu\in\RR$ and $\xi\in S^{N-1}$.

\item[(2)] If $a_2(x)$ satisfies that $\max_{x\in\RR^N}a_2(x)-\min_{x\in\RR^N}a_2(x)<\inf_{\xi\in S^{N-1}}\int_{z \cdot\xi \leq 0}k(z)dz$, then the principal eigenvalue $\lambda_2(\mu,\xi,a_2)$ of $\mathcal{O}_{2,\mu,\xi}$ exists for
all $\mu\in\RR$ and $\xi\in S^{N-1}$.
\end{itemize}
\end{proposition}

\begin{proof}
(1) It follows from \cite[Theorem B]{ShZh1}.

(2) It follows from \cite[Theorem B$^{'}$]{ShZh2}.
\end{proof}

Let $\hat a_i$  be the average of $a_i(\cdot)$ ($i=1,2,3$), that is,
\vspace{-0.05in}\begin{equation}
\label{average-eq}
\begin{cases}
\hat a_i=\frac{1}{|D_i|}\int_{D_i} a_i(x)dx\quad {\rm for}\quad i=1,2\cr
\hat a_3=\frac{1}{\# D_3}\sum_{j\in D_3}a_3(j),
\end{cases}
\vspace{-0.05in}\end{equation}
where
\vspace{-0.05in}\begin{equation}
\label{domain-eq}
D_i=[0,p_1]\times [0,p_2]\times\cdots\times[0,p_N]\cap
\mathcal{H}_i,\,\,  i=1,2,3
\vspace{-0.05in}\end{equation}
and
\vspace{-0.05in}\begin{equation}
\label{size-domain-eq}
 \begin{cases}
 |D_i|=p_1\times
p_2\times\cdots\times p_N\,\, {\rm for}\,\,   i=1,2\cr
 \# D_3=\,\text{the
cardinality of}\,\, D_3.
\end{cases}
\vspace{-0.05in}\end{equation}
 By Proposition
\ref{existence-principal-eigenvalue} (2), $\lambda_2(\mu,\xi,\hat
a_2)$ exists for all $\mu\in\RR$ and $\xi\in S^{N-1}$. The following
proposition shows a relation between $\lambda_i^0(\mu,\xi,a_i)$ and
$\lambda_i^0(\mu,\xi,\hat a_i)$.

\begin{proposition}[Influence of spatial variation]
\label{lower-bound-pe}
For given $1\leq i\leq 3$, $\mu\in\RR$, and $\xi\in S^{N-1}$, there holds
\vspace{-0.05in}$$
\lambda_i^0(\mu,\xi,a_i)\geq \lambda_i^0(\mu,\xi,\hat a_i).
\vspace{-0.05in}$$
\end{proposition}

\begin{proof}
The case $i=1$ is well known.
The cases $i=2$ and $3$ follow from \cite[Theorem 2.1]{HeShZh}.
\end{proof}

We remark that $\lambda_i(\mu,\xi,\hat a_i)(=\lambda_i^0(\mu,\xi,\hat a_i))$ ($i=1,2,3$) have the following
explicit expressions,
\vspace{-0.05in}\begin{equation}
\label{pe-expression-eq}
\begin{cases}
\lambda_1(\mu,\xi,\hat a_1)=\hat a_1+\mu^2\cr
\lambda_2(\mu,\xi,\hat a_2)=\int_{\RR^N} e^{-\mu z\cdot\xi}\kappa(z)dz-1+\hat a_2\cr
\lambda_3(\mu,\xi,\hat a_3)=\sum_{k\in K}a_k(e^{-\mu k\cdot\xi}-1)+\hat a_3.
\end{cases}
\vspace{-0.05in}\end{equation}

\subsection{KPP equations in spatially periodic media}

In this subsection, we recall some spatial spreading dynamics of KPP equations in
spatially periodic media.

Consider
\vspace{-0.05in}\begin{equation}
\label{main-periodic-eq}
u_t(t,x)=\Delta u(t,x)+ u(t,x)g_1(x,u(t,x)),\quad x\in\RR^N,
\vspace{-0.05in}\end{equation}
\vspace{-0.05in}\begin{equation}
\label{nonlocal-periodic-eq}
u_t(t,x)=\int_{\RR^N} \kappa(y-x)u(t,y)dy-u(t,x)+u(t,x)g_2(x,u(t,x)),\quad x\in\RR^N,
\vspace{-0.05in}\end{equation}
and
\vspace{-0.05in}\begin{equation}
\label{discrete-periodic-eq}
u_t(t,j)=\sum_{k\in K}a_k(u(t,j+k)-u(t,j))+u(t,j)g_3(j,u(t,j)),\quad j\in\ZZ^N,
\vspace{-0.05in}\end{equation}
where $g_i(\cdot,\cdot)$ $(i=1,2,3$) are periodic in the first variable and monostable in the second
variable. More precisely, we assume

\smallskip

\noindent{\bf (P1)} {\it  $1\leq i\leq 3$ and $g_i:\mathcal{H}_i\times\RR\to\RR$ is a
$C^2$ function, $g_i(x+p_l{\bf e}_l,u)=g_i(x,u)$, where $p_l>0$ and $p_l\in \NN$ in the case $i=3$
$(l=1,2,\cdots,N)$,
and
 $g_i(x,u)<0$  for all
$(x,u)\in\mathcal{H}_i\times\RR^+$ with $u\ge \alpha_0$ for some $\alpha_0>0$ and  $\p_ug_i(x,u)<0$ for
all $(x,u)\in\mathcal{H}_i\times\RR^+$.}

\smallskip
\noindent{\bf (P2)} {\it  $\lambda_i^0(g_i(\cdot,0))>0$,  where $i=1,2,3$.
}

\smallskip

Assume (P1). Similarly, by general semigroup theory, for any $u_0\in X_1$ (resp. $u_0\in X_2$, $u_0\in X_3$),
\eqref{main-periodic-eq} (resp. \eqref{nonlocal-periodic-eq}, \eqref{discrete-periodic-eq}) has a
unique (local) solution $u_1(t,\cdot;u_0,g_1(\cdot,\cdot))(\in X_1)$ (resp. $u_2(t,\cdot;u_0,g_2(\cdot,\cdot))(\in X_2)$, $u_3(t,\cdot;u_0,g_3(\cdot,\cdot))(\in X_3)$)
with initial data $u_0(\cdot)$. Moreover, if $u_0\in X_{i,p}$, then
$u_i(t,\cdot;u_0,g_i(\cdot,\cdot))\in X_{i,p}$ for any $t>0$ at which $u_i(t,\cdot;u_0,g_i(\cdot,\cdot))$ exists ($i=1,2,3$).
By Proposition \ref{basic-comparison}, if $u_0\in X_i^+$, then $u_i(t,\cdot;u_0,g_i(\cdot,\cdot))$ exists and
$u_i(t,\cdot;u_0,g_i(\cdot,\cdot))\in X_i^+$ for all $t>0$ ($i=1,2,3$).

\begin{proposition} [Spatially periodic positive stationary solution]
\label{periodic-prop1} Assume (P1) and (P2). Then
\eqref{main-periodic-eq} $($resp. \eqref{nonlocal-periodic-eq},
\eqref{discrete-periodic-eq}$)$ has a unique spatially periodic
stationary solution $u_1^*(\cdot;g_1(\cdot,\cdot))\in X_{1,p}^{++}$
$($resp. $u_2^*(\cdot;g_2(\cdot,\cdot))\in X_{2,p}^{++}$,
$u_3^*(\cdot;g_3(\cdot,\cdot))\in X_{3,p}^{++})$ which is globally
asymptotically stable with respect to perturbations in
$X_{1,p}^+\setminus\{0\}$ $($resp. $X_{2,p}^+\setminus\{0\}$,
$X_{3,p}^+\setminus\{0\})$.
\end{proposition}

\begin{proof}
The cases that $i=1$ and $3$ follow from \cite[Theorem 2.3]{Zha1}.
The case that $i=2$ follows from \cite[Theorem C]{ShZh2}.
\end{proof}

\begin{proposition} [Spreading speeds]
\label{periodic-prop2} Assume (P1) and (P2). Then for any $\xi\in
S^{N-1}$, \eqref{main-periodic-eq} $($resp.
\eqref{nonlocal-periodic-eq}, \eqref{discrete-periodic-eq}$)$ has a
spreading speed $c_{1}^*(\xi;g_1(\cdot,\cdot))$ $($resp.
$c_{2}^*(\xi;g_2(\cdot,\cdot))$, $c_{3}^*(\xi; g_3(\cdot,\cdot)))$
in the direction of $\xi$. Moreover,
\vspace{-0.05in}$$
c_{i}^*(\xi;g_i(\cdot,\cdot))=\inf_{\mu>0}\frac{\lambda_i^0(\mu,\xi,g_i(\cdot,0))}{\mu} \quad (i=1,2,3)
\vspace{-0.05in}$$
and the following hold for $i=1,2,3$.
\begin{itemize}
\vspace{-0.05in}\item[(1)] For each $ u_0\in X_i^+$ satisfying that  $u_0(x)=0$ for $x\in\mathcal{H}_i$ with $|x\cdot\xi|\gg 1$,
\vspace{-0.05in}$$
\limsup_{|x\cdot\xi|\geq ct, t\to\infty}u_i(t,x;u_0,g_i(\cdot,\cdot))=0\quad \forall c>\max\{c^*_i(\xi;g_i(\cdot,\cdot)),c^*_i(-\xi;g_i(\cdot,\cdot))\}.
\vspace{-0.05in}$$

\vspace{-0.05in}\item[(2)]  For each $\sigma>0$,  $r>0$, and   $u_0\in X^+_i$ satisfying that $u_0(x)\geq\sigma$ for  $x\in\mathcal{H}_i$ with $|x\cdot \xi|\leq r$,
\vspace{-0.05in}$$
\limsup_{|x\cdot\xi|\leq ct, t\to\infty}|u_i(t,x;u_0,g_i(\cdot,\cdot))-u_i^*(x;g_i(\cdot,\cdot))|=0
\vspace{-0.05in}$$
for all
$0<c<\min\{c^*_i(\xi;g_i(\cdot,\cdot)),c^*_i(-\xi;g_i(\cdot,\cdot))\}$.

\vspace{-0.05in}\item[(3)] For each  $u_0\in X_i^+$ satisfying that $u_0(x)=0$ for $x\in\mathcal{H}_i$ with $\|x\|\gg 1$,
\vspace{-0.05in}$$
\limsup_{\|x\|\geq ct,t\to\infty}u_i(t,x;u_0,g_i(\cdot,\cdot))=0\quad \forall c>\sup_{\xi\in S^{N-1}}c_i^*(\xi;g_i(\cdot,\cdot)).
\vspace{-0.05in}$$

\vspace{-0.05in}\item[(4)] For each $\sigma>0$, $r>0$, and  $u_0\in X_i^+$ satisfying that $u_0(x)\geq\sigma$ for $x\in\mathcal{H}_i$ with  $\|x\|\leq r$,
\vspace{-0.05in}$$
\limsup_{\|x\|\leq ct,t\to\infty}|u_i(t,x;u_0,g_i(\cdot,\cdot))-u_i^*(x;g_i(\cdot,\cdot))|=0 \quad\forall  0<c<\inf_{\xi\in S^{N-1}}c_i^*(\xi;g_i(\cdot,\cdot)).$$
\end{itemize}
\end{proposition}

\begin{proof}
The cases $i=1$ and $i=3$ follow from \cite[Theorems 3.1-3.4 and Corollary 3.1]{LiZh2} (see also \cite[Theorems 1.2-2.3]{Wei2}) and the case $i=2$ follows from  \cite[Theorems D and E]{ShZh2}.
\end{proof}

Let $\hat g_1(u)$ $($resp. $\hat g_2(u)$, $\hat g_3(u)$) be the
spatial average of $g_1(x,u)$ $($resp. $g_2(x,u)$, $g_3(x,u)$),
respectively, that is,
\begin{equation}
\label{average-g-eq}
\begin{cases}
\hat g_i(u)=\frac{1}{|D_i|}\int_{D_i} g_i(x,u)dx\quad {\rm for}\quad i=1,2\cr
\hat g_3(u)=\frac{1}{\# D_3}\sum_{j\in D_3}g_3(j,u),
\end{cases}
\end{equation}
where $D_i$ ($i=1,2,3$), $|D_i|$ ($i=1,2$) and $\# D_3$ are as in \eqref{domain-eq} and \eqref{size-domain-eq}.

Assume

\smallskip
\noindent{\bf (P3)} $\hat g_i(0)>0$ ($i=1,2,3$).
\smallskip

 Observe that $\lambda_i(\hat g_i(0))=\hat g_i(0)$. Then by Proposition \ref{lower-bound-pe},
 (P3) implies (P2).
Assume (P3). By Proposition \ref{periodic-prop2}, for any $\xi\in S^{N-1}$,
 \eqref{main-periodic-eq} $($resp. \eqref{nonlocal-periodic-eq}, \eqref{discrete-periodic-eq}) with $g_1(x,u)$ $($resp. $g_2(x,u)$, $g_3(j,u)$)
 being replaced by $\hat g_1(u)$ $($resp. $\hat g_2(u)$, $\hat g_3(u)$)
 has a spreading speed
 $ c_{1}^*(\xi;\hat g_1(\cdot))$ $($resp. $ c_{2}^*(\xi;\hat g_2(\cdot))$, $c_{3}^*(\xi;\hat g_3(\cdot))$)   in the direction of $\xi\in S^{N-1}$.

\begin{proposition} [Influence of spatial variation]
\label{periodic-prop3}
Assume (P1) and (P3). Then for any $\xi\in S^{N-1}$,
\vspace{-0.05in}$$c_{i}^*(\xi;g_i(\cdot,\cdot))\geq  c_{i}^*(\xi;\hat g_i(\cdot)),\quad i=1,2,3.
\vspace{-0.05in}$$
\end{proposition}
\begin{proof}
Let $a_i(\cdot)=g_i(\cdot,0)$.
By Proposition \ref{periodic-prop2},
\vspace{-0.05in}$$
c_{i}^*(\xi;g_i(\cdot,\cdot))=\inf_{\mu>0}\frac{\lambda_i^0(\mu,\xi,a_i)}{\mu}
\quad {\rm and}\quad c_{i}^*(\xi;\hat g_i(\cdot))=\inf_{\mu>0}\frac{\lambda_i^0(\mu,\xi,\hat a_i)}{\mu}
\vspace{-0.05in}$$
for $i=1,2,3$. By Proposition \ref{lower-bound-pe},
\vspace{-0.05in}$$
\lambda_i^0(\mu,\xi,a_i)\geq \lambda_i^0(\mu,\xi,\hat a_i)\quad i=1,2,3.
\vspace{-0.05in}$$
The proposition then follows.
\end{proof}

\section{Positive Stationary Solutions  and the Proof of Theorem 2.1}

In this section, we investigate the existence of positive stationary solutions of \eqref{main-eq}, \eqref{nonlocal-eq},
and \eqref{discrete-eq}, and prove Theorem 2.1.

Throughout this section, we assume ${\rm (H1)}$ and ${\rm (H2)}$.  We first prove some lemmas.

\begin{lemma}
\label{tech-lm1}
For any $1\le i\le 3$ and $\epsilon>0$, there are $p=(p_1,p_2,\cdots,p_N)\in\NN^N$ and
 $h_i\in X_{i,p}\cap C^N(\mathcal{H}_i,\RR)$   such that
\vspace{-0.05in}$$
f_i(x,0)\geq h_i(x)\quad {\rm for}\quad x\in\mathcal{H}_i,
\vspace{-0.05in}$$
\vspace{-0.08in}$$
\hat h_i\geq f_i^0(0)-\epsilon \quad {\rm (hence}\quad
\lambda_i^0(h_i(\cdot))\geq f_i^0(0)-\epsilon),
\vspace{-0.05in}$$
and for the cases that $i=1$ and $2$,
the partial derivatives of $h_i(x)$ up to order $N-1$ are zero at some $x_0\in\mathcal{H}_i$ with
$h_i(x_0)=\max_{x\in\mathcal{H}_i}(x)$, where
$\hat h_i$ is the average of $h_i(\cdot)$ $($see \eqref{average-eq} for the definition$)$.
\end{lemma}

\begin{proof} Fix $1\le i\le 3$.
By ${\rm (H2)}$, there is $L_0>0$  such that $f_i(x,0)=f_i^0(0)$ for
$x\in\mathcal{H}_i$ with  $\|x\|\geq L_0$. Let
$M_0=\inf_{x\in\mathcal{H}_i,1\le i\le 3}f_i(x,0)$. Let $h_0:\RR\to [0,1]$ be a
smooth function such that $h_0(s)=1$ for $|s|\leq 1$ and $h_0(s)=0$
for $|s|\geq 2$. For any $p=(p_1,p_2,\cdots,p_N)\in\NN^N$ with
$p_j>4L_0$, let $h_i\in X_{i,p}\cap C^N(\mathcal{H}_i,\RR)$ $(i=1,2,3)$  be such that
\vspace{-0.05in}$$h_i(x)=f_i^0(0)-h_0\big(\frac{\|x\|^2}{L_0^2}\big)(f_i^0(0)-M_0)\quad {\rm for}\quad x\in \Big([-\frac{p_1}{2},\frac{p_1}{2}]\times[-\frac{p_2}{2},\frac{p_2}{2}]\times\cdots\times [-\frac{p_N}{2},
\frac{p_N}{2}]\Big)\cap \mathcal{H}_i.
\vspace{-0.05in}$$
Then
\vspace{-0.05in}$$
f_i(x,0)\geq h_i(x)\quad \forall x\in\mathcal{H}_i,\,\, 1\le i\le 3.
\vspace{-0.05in}$$
It is clear that for $i=1$ or $2$, the partial derivatives of $h_i(x)$ up to order $N-1$ are zero
at some $x_0\in\mathcal{H}_i$ with
$h_i(x_0)=\max_{x\in\mathcal{H}_i}h_i(x)(=f_i^0(0))$.
For given $\epsilon>0$,
choosing  $p_j\gg 1$, we have
\vspace{-0.05in}$$\hat h_i>f_i^0(0)-\epsilon.
\vspace{-0.05in}$$
By Proposition \ref{lower-bound-pe},
$\lambda_i^0(h_i(\cdot))\geq \lambda_i^0(\hat h_i)=\hat h_i$
and hence
\vspace{-0.05in}$$
\lambda_i^0(h_i(\cdot))\geq f_i^0(0)-\epsilon.
\vspace{-0.05in}$$
The lemma is thus proved.
\end{proof}

\begin{lemma}
\label{tech-lm0}
Suppose that $\tilde u_2^*:\RR^N\to [\sigma_0,M_0]$ is
Lebesgue measurable, where  $\sigma_0$ and $M_0$ are two positive constants. If
\vspace{-0.05in}$$
\int_{\RR^N}\kappa(y-x)\tilde u_2^*(y)dy-\tilde u_2^*(x)+\tilde u_2^*(x)\tilde f_2(x,\tilde u_2^*(x))=0\quad \forall x\in\RR^N,
\vspace{-0.05in}$$
where $ \tilde f_2(x,u)=f_2(x,u)$ or $f_2^0(u)$ for all $x\in\RR^N$ and $u\in\RR$, then $\tilde u_2^*(\cdot)\in X_2^{++}$.
\end{lemma}

\begin{proof}
We prove the case that $\tilde f_2(x,u)=f_2(x,u)$. The case that $\tilde f_2(x,u)=f_2^0(u)$ can be proved similarly.

Let $h^*(x)=\int_{\RR^N}\kappa(y-x)\tilde u_2^*(y)dy$ for $x\in\RR^N$. Then  $h^*(\cdot)$ is $C^1$ and
 has bounded first order partial derivatives. Let
\vspace{-0.05in}$$
 F(x,\alpha)=h^*(x)-\alpha+\alpha f_2(x,\alpha)\quad \forall x\in\RR^N,\,\, \alpha\in\RR.
 \vspace{-0.05in}$$
 Then $F:\RR^N\times\RR\to \RR$ is $C^1$ and $F(x,\tilde u_2^*(x))=0$ for each $x\in\RR^N$. If $\alpha^*>0$ is such that $F(x,\alpha^*)=0$,
 then
 \vspace{-0.05in}$$
 -1+f_2(x,\alpha^*)=-\frac{h^*(x)}{\alpha^*}<0
 \vspace{-0.05in}$$
 and hence
 \vspace{-0.05in}$$
 \p_\alpha F(x,\alpha^*)=-1+f_2(x,\alpha^*)+\alpha^*\p_u f_2(x,\alpha^*)<0.
 \vspace{-0.05in}$$
 By Implicit Function Theorem, $\tilde u_2^*(x)$ is $C^1$ in $x$. Moreover,
 \vspace{-0.05in}$$
 \frac{\p \tilde u_2^*(x)}{\p x_j}=\frac{\frac{\p h^*(x)}{\p x_j}}{-1+f(x,\tilde u_2^*(x))+\p_u f_2(x,\tilde u_2^*(x))\tilde u_2^*(x)}\quad \forall x\in\RR^N,\,\, 1\le j\le N.
 \vspace{-0.05in}$$
 Therefore, $\tilde u_2^*$ has  bounded first order partial derivatives. It
 then follows that $\tilde u_2^*(x)$ is uniformly continuous in
 $x\in\RR^N$ and then $\tilde u_2^*\in X_2^{++}$.
\end{proof}

\begin{lemma}
\label{tech-lm2}
Suppose that  $u^*_i(\cdot)\in X_i^{++}$ and  $u=u_i^*(\cdot)$ is a
stationary solution of (1.i) $(1\leq i\leq 3$).
Then
\vspace{-0.05in}$$
u_i^*(x)\to u_i^0\quad {\rm as}\quad \|x\|\to \infty.
\vspace{-0.05in}$$
\end{lemma}

\begin{proof}
We first prove that
 \vspace{-0.05in}$$
 u_1^*(x)\to u_1^0\quad {\rm as}\quad \|x\|\to\infty.
 \vspace{-0.05in}$$
 Assume that $u_1^*(x)\not \to u_1^0$ as $\|x\|\to\infty$.  Then there are $\epsilon_0>0$ and  $x_n\in\RR^N$ such that $\|x_n\|\to \infty$ and
\vspace{-0.05in}$$|u_1^*(x_n)- u_1^0|\geq \epsilon_0\quad {\rm for}\quad n=1,2,\cdots.
\vspace{-0.05in}$$
 By the uniform continuity of $u_1^*(x)$ in $x\in\RR^N$, without loss of generality,
 we may assume that there is a continuous function $\tilde u_1^*:\RR^N\to [\sigma_0,M_0]$ for some $\sigma_0,M_0>0$   such that
\vspace{-0.05in} $$
 u_1(x+x_{n})\to \tilde u_1^*(x)
 \vspace{-0.05in}$$
 as $n\to\infty$ uniformly in $x$  on bounded sets. Moreover, by a priori estimates for parabolic equations, $\tilde u_1^*$ is $C^{2+\alpha}$ for some
 $\alpha>0$  and we may also assume that
 \vspace{-0.05in}$$
 \Delta u_1(x+x_{n})\to \Delta \tilde u_1^*(x)
\vspace{-0.05in} $$
 as $n\to\infty$ uniformly in $x$  on bounded sets.
 This together with $f_1(x+x_n,u)\to f_1^0(u)$ as $n\to\infty$ uniformly in $x$ on bounded sets and in $u\in\RR$ implies that
 \vspace{-0.05in}$$
 \Delta \tilde u_1^*+\tilde u_1^* f_1^0(\tilde u_1^*)=0,\quad x\in\RR^N.
 \vspace{-0.05in}$$
 By Proposition \ref{periodic-prop1}, we  must have $\tilde u_1^*(x)\equiv u_1^*(x;f_1^0(\cdot))\equiv u_1^0$ and hence
 $u_1^*(x_{n})\to u_1^0$ as $n\to\infty$. This is a contradiction. Therefore  $u_1^*(x)\to u_1^0$ as $\|x\|\to \infty$.

 Next, we prove that
 \vspace{-0.05in}$$
 u_2^*(x)\to u_2^0\quad {\rm as}\quad \|x\|\to\infty.
 \vspace{-0.05in}$$
 Similarly, assume that $u_2^*(x)\not \to u_2^0$ as $\|x\|\to\infty$.  Then there are $\epsilon_0>0$ and  $x_n\in\RR^N$ such that $\|x_n\|\to \infty$ and
\vspace{-0.05in}$$|u_2^*(x_n)- u_2^0|\geq \epsilon_0\quad {\rm for}\quad n=1,2,\cdots.
\vspace{-0.05in}$$
 By the uniform continuity of $u_2^*(x)$ in $x\in\RR^N$, without loss of generality,
 we may assume that there is a continuous function $\tilde u_2^*:\RR^N\to [\sigma_0,M_0]$ for some $\sigma_0,M_0>0$   such that
\vspace{-0.05in} $$
 u_2(x+x_{n})\to \tilde u_2^*(x)
 \vspace{-0.05in}$$
 as $n\to\infty$ uniformly in $x$  on bounded sets.
 By the Lebesgue Dominated Convergence Theorem, we have
 \vspace{-0.05in}$$
 \int_{\RR^N}\kappa(y-x)\tilde u_2^*(y)dy-\tilde u_2^*(x)+\tilde u_2^*(x)f_2^0(\tilde u_2^*(x))=0\quad \forall x\in\RR^N.
 \vspace{-0.05in}$$
 By Lemma \ref{tech-lm0}, $\tilde u_2^*\in X_2^{++}$.
 By Proposition \ref{periodic-prop1} again, we have $\tilde u_2^*(x)\equiv u_2^0$ and then
 $u_2^*(x_{n})\to u_2^0$ as $n\to\infty$. This is a contradiction. Therefore  $u_2^*(x)\to u_2^0$ as $\|x\|\to \infty$.

 Finally, it can be proved by the similar arguments as in the case $i=2$ that
 \vspace{-0.05in}$$
 u_3^*(j)\to u_3^0\quad {\rm as}\quad \|j\|\to\infty.
 \vspace{-0.05in}$$
\end{proof}

\begin{lemma}
\label{tech-lm3} There is $u_i^-\in X_i^{++}$ such that for any
$\delta>0$ sufficiently small,  $u_i(t,x;\delta u_i^-)$ is
increasing in $t>0$ and $u_i^{-,*,\delta}\in X_i^{++}$, where
$u_i^{-,*,\delta}(x)=\lim_{t\to\infty}u_i(t,x;\delta u_i^-)$, and
hence $u=u_i^{-,*,\delta}(\cdot)$ is a stationary solution of (1.i) in $X_i^{++}$  $(i=1,2,3)$.
\end{lemma}

\begin{proof} Fix $1\le i\le 3$.
Let $M^*>0$ be such that $f_i(x,M^*)<0$. Let $\epsilon>0$ be such that
\vspace{-0.05in}$$
f_i^0(0)-\epsilon>0.
\vspace{-0.05in}$$
By  Lemma \ref{tech-lm1}, there are $p\in\NN^N$ and $h_i(\cdot)\in X_{i,p}\cap C^N(\mathcal{H}_i,\RR)$ such that
\vspace{-0.05in}$$
f_i(x,0)\ge h_i(x),\,\,\, {\rm and}\,\,\,
\hat h_i\ge f_i^0(0)-\epsilon(>0).
\vspace{-0.05in}$$
Moreover, for $i=1$ or $2$,  the partial derivatives of $h_i(x)$ up to order $N-1$ are zero at some $x_0\in\mathcal{H}_i$ with
$h_i(x_0)=\max_{x\in\mathcal{H}_i}h_i(x)$.
Let $u_i^-$ be the positive principal eigenfunction of  $\mathcal{O}_{i,0,0}$ with $a_i(\cdot)=h_i(\cdot)$   and
$\|u_i^-\|=1$ (the existence of
$u_i^-$ is well known in the case that $i=1$ or $3$ and follows from  Proposition \ref{existence-principal-eigenvalue} in the case
that $i=2$). It is not difficult to verify that  $u=\delta u_i^-$ is a sub-solution of (1.i) for any $\delta>0$ sufficiently small. It then follows that
for any $\delta>0$ sufficiently small,
\vspace{-0.05in}$$
\delta u_i^-(\cdot)\le u_i(t_1,\cdot;\delta u_i^-)\leq u_i(t_2,\cdot;\delta u_i^-)\quad\forall 0<t_1<t_2.
\vspace{-0.05in}$$
This implies that there is a Lebesgue measurable function $u_i^{-,*,\delta}:\mathcal{H}_i\to [\sigma_0, M_0]$
for some $\sigma_0,M_0>0$  such that
\vspace{-0.05in}$$
\lim_{t\to\infty}u_i(t,x;\delta u_i^-)=u_i^{-,*,\delta}(x)\quad\forall x\in\mathcal{H}_i.
\vspace{-0.05in}$$
Moreover, by regularity and a priori estimates for parabolic
equations, $u_1^{-,*,\delta}\in X_1^{++}$. It is clear that
$u_3^{-,*,\delta}\in X_3^{++}$.  By Lemma \ref{tech-lm0},   $u_2^{-,*,\delta}\in
X_2^{++}$. Therefore for $1\leq i\leq 3$, $u_i^{-,*,\delta}\in
X_i^{++}$ and $u=u_i^{-,*,\delta}(\cdot)$  is a stationary solution of (1.i) in $X_i^{++}$ ($i=1,2,3$).
\end{proof}

\begin{lemma}
\label{tech-lm4} Let $M\gg 1$ be such that $f_i(x,M)<0$ for
$x\in\mathcal{H}_i$ $(i=1,2,3)$. Then $\lim_{t\to\infty}u_i(t,x;u_0)$ exists for every
$x\in\mathcal{H}_i$, where $u_0(x)\equiv M$. Moreover,
$u_i^{+,*,M}(\cdot)\in X_i^{++}$, where
$u_i^{+,*,M}(x):=\lim_{t\to\infty}u_i(t,x;u_0)$,  and hence $u=u_i^{+,*,M}(\cdot)$ is a
stationary solution of (1.i) in $X_i^{++}$ $(i=1,2,3)$.
\end{lemma}

\begin{proof} Fix $1\le i\le 3$.
For any $M>1$ with $f_i(x,M)<0$ for all $x\in\mathcal{H}_i$,
$u=M$ is a super-solution of (1.i). Hence
\vspace{-0.05in}$$
u_i(t_2,\cdot;M)\leq u_i(t_1,\cdot;M)\leq M\quad\forall 0\leq t_1<t_2.
\vspace{-0.05in}$$
It then follows that $\lim_{t\to\infty}u_i(t,x;M)$ exists for all
$x\in\RR^N$. Let $u_i^{+,*,M}(x)=\lim_{t\to\infty}u_i(t,x;M)$. We
have $u_i^{+,*,M}(x)\geq u_i^{-,*,\delta}(x)$ for $0<\delta\ll 1$.
By the similar arguments as in Lemma \ref{tech-lm3}, $u_i^{+,*,M}\in
X_i^{++}$ and $u=u_i^{+,*,M}(\cdot)$ is a stationary solution of
(1.i) in $X_i^{++}$ $(i=1,2,3)$.
\end{proof}

\begin{proof} [Proof of Theorem 2.1]
(1) Let $1\leq i\leq 3$ be given. First, by Lemmas \ref{tech-lm3} and \ref{tech-lm4}, (1.i) has
stationary solutions  in $X_i^{++}$. We claim that stationary
solution of (1.i) in $X_i^{++}$ is unique. In fact, suppose that
$u_i^{1,*}$ and $u_i^{2,*}$ are two stationary solutions of (1.i) in
$X_i^{++}$. Assume that
$u_i^{1,*}\not= u_i^{2,*}$. Then there is $\alpha^*>1$ such that
$\rho_i(u_i^{1,*},u_i^{2,*})=\ln\alpha^*>0$. Note that
\vspace{-0.05in}$$
\frac{1}{\alpha^*}u_i^{1,*}\leq u_i^{2,*}\leq\alpha^* u_i^{1,*}.
\vspace{-0.05in}$$
 By Lemma \ref{tech-lm2}, $\lim_{\|x\|\to
\infty}u_i^{1,*}(x)=u_i^0$ and
$\lim_{\|x\|\to\infty}u_i^{2,*}(x)=u_i^0$.
This implies that there is $\epsilon>0$ such that
\vspace{-0.05in}$$
\frac{1}{\alpha^*-\epsilon}u_i^{1,*}(x)\leq u_i^{2,*}(x)\leq (\alpha^*-\epsilon) u_i^{1,*}(x)\quad {\rm for}\quad \|x\|\gg 1.
\vspace{-0.05in}$$
By Proposition \ref{basic-comparison} and the arguments in Proposition \ref{basic-part-metric},
\vspace{-0.05in}$$
\frac{1}{\alpha^*}u_i^{1,*}(x)<u_i^{2,*}(x)<\alpha^* u_i^{1,*}(x)\quad \forall x\in\RR^N.
\vspace{-0.05in}$$
It then follows that for $0<\epsilon\ll 1$,
\vspace{-0.05in}$$
\frac{1}{\alpha^*-\epsilon}u_i^{1,*}(x)\leq u_i^{2,*}(x)\leq (\alpha^*-\epsilon) u_i^{1,*}(x)\quad \forall x\in\RR^N
\vspace{-0.05in}$$
and then $\rho_i(u_i^{1,*},u_i^{2,*})\leq \ln (\alpha^*-\epsilon)$, this is a contradiction.
Therefore $u_i^{1,*}=u_i^{2,*}$ and  (1.i) has a unique stationary solution $u_i^*$ in $X_i^{++}$.

(2) Fix $1\le i\le 3$.  For any $u_0\in X_i^{++}$, there is $\delta>0$ sufficiently
small and $M>0$ sufficiently large such that $\delta u_i^-\leq
u_0\leq M$ and $u=\delta u_i^-$ is a sub-solution of (1.i) ($u_i^-$ is as in Lemma \ref{tech-lm3}) and $u=M$
is a super-solution of (1.i). Then
\vspace{-0.05in}$$
\delta u_i^-\leq u_i(t,\cdot;\delta u_i^-)\leq u_i(t,\cdot;u_0)\leq u_i(t,\cdot;M)\leq M\quad \forall t\geq 0.
\vspace{-0.05in}$$
By (1), Lemmas \ref{tech-lm3} and \ref{tech-lm4}, and Dini's Theorem,
\vspace{-0.05in}$$
u_i(t,x;\delta u_i^-)<u_i^*(x)<u_i(t,x;M)\quad \forall t>0,\,\,
x\in\mathcal{H}_i
\vspace{-0.05in}$$
and
\vspace{-0.05in}$$
\lim_{t\to\infty}u_i(t,x;\delta
u_i^-)=\lim_{t\to\infty}u_i(t,x;M)=u_i^*(x)
\vspace{-0.05in}$$
uniformly in $x$ on bounded sets. It then follows that
\vspace{-0.05in}$$
\lim_{t\to\infty}u_i(t,x;u_0)=u_i^*(x)
\vspace{-0.05in}$$
uniformly in $x$ on bounded sets.

We claim that  $\|u_i(t,\cdot;u_0)-u_i^*(\cdot)\|\to 0$ as
$t\to\infty$.  Assume the claim is not true. Then  there are $\epsilon_0>0$,
$t_n\to\infty$, and $x_n$ with $\|x_n\|\to \infty$ such that
\vspace{-0.05in}$$|u_i(t_n,x_n;u_0)-u_i^*(x_n)|\geq \epsilon_0\quad \forall n\in\NN.
\vspace{-0.05in}$$
 Then by Lemma \ref{tech-lm2},
\vspace{-0.05in}$$|u_i(t_n,x_n;u_0)-u_i^0|\geq \frac{\epsilon_0}{2} \quad \forall n\gg 1.
\vspace{-0.05in}$$
 Let $\tilde
\delta>0$ and $\tilde M>0$  be such that
\vspace{-0.05in}$$
\tilde \delta\leq u_i(t,\cdot;u_0)\leq\tilde  M\quad \forall t\geq
0.
\vspace{-0.05in}$$
For any $\epsilon>0$, let $T>0$ be such that
\vspace{-0.05in}\begin{equation}
\label{thm2-1-eq1}
|u_i(T,\cdot;\tilde\delta,f_i^0(\cdot))-u_i^0|<\epsilon,\quad
|u_i(T,\cdot;\tilde M,f_i^0(\cdot))-u_i^0|<\epsilon.
\vspace{-0.05in}\end{equation}
Observe that
\vspace{-0.05in}$$\tilde \delta\leq u_i(t_n-T,x_n+x;u_0)\leq\tilde M
\vspace{-0.05in}$$
and
\vspace{-0.05in}$$
u_i(t_n,x_n+\cdot;u_0)=u_i(T,x_n+\cdot;u_i(t_n-T,\cdot;u_0))=
u_i(T,\cdot;u_i(t_n-T,\cdot+x_n;u_0),f_i(\cdot+x_n,\cdot))
\vspace{-0.05in}$$
for $n\gg 1$.
Then
\vspace{-0.05in}\begin{equation}
\label{thm2-1-eq2}
u_i(T,\cdot;\tilde \delta,f_i(\cdot+x_n))\leq u_i(t_n,x_n+\cdot;u_0)\leq
u_i(T,\cdot;\tilde M,f_i(\cdot+x_n,\cdot)).
\vspace{-0.05in}\end{equation}
Observe also that $f_i(x+x_n,u)\to f_i^0(u)$ as $n\to\infty$ uniformly in $(x,u)$ on bounded
sets. Then
by Proposition \ref{basic-convergence},
\vspace{-0.05in}$$
u_i(T,x;\tilde \delta,f_i(\cdot+x_n,\cdot))\to u_i(T,x;\tilde \delta,f_i^0(\cdot))
\vspace{-0.05in}$$
and
\vspace{-0.05in}$$
u_i(T,x;\tilde M,f_i(\cdot+x_n,\cdot))\to u_i(T,x;\tilde M,f_i^0(\cdot))
\vspace{-0.05in}$$
as $n\to\infty$ uniformly in $x$ on bounded sets.
This together with \eqref{thm2-1-eq1} implies that
\vspace{-0.05in}$$
|u_i(T,0;\tilde \delta,f_i(\cdot+x_n,\cdot))-u_i^0|<2\epsilon,\quad |u_i(T,0;\tilde M,f_i(\cdot+x_n,\cdot))-u_i^0|<2\epsilon\quad {\rm for}\quad n\gg 1
\vspace{-0.05in}$$
and then by \eqref{thm2-1-eq2},
\vspace{-0.05in}$$
|u_i(t_{n},x_{n};u_0)-u_i^0|<2\epsilon\quad {\rm for}\quad n\gg 1.
\vspace{-0.05in}$$
Hence $\lim_{n\to\infty}u_i(t_n,x_n;u_0)=u_i^0$,
which is a contradiction. Therefore
$\|u_i(t,\cdot;u_0)-u_i^*(\cdot)\|\to 0$ as $t\to\infty$.

(3) By Proposition \ref{basic-comparison},  for any $u_0\in
X_i^+\setminus\{0\}$, $u_i(t,x;u_0)>0$ for all $t>0$ and
$x\in\mathcal{H}_i$. Hence for any given $u_0\in
X_i^+\setminus\{0\}$, there are $\sigma>0$ and $r>0$ such that
$u_i(1,x;u_0)\ge \sigma$ for $x\in\mathcal{H}_i$ with $\|x\|\leq r$.
Note that $u_i(t,\cdot;u_0)=u_{i}(t-1,\cdot;u_i(1,\cdot;u_0))$ for
$t\ge 1$. (3) then follows from Theorem 2.3 (4) (see next section
for the proof of Theorem 2.3 (4)).
\end{proof}

\section{Spatial Spreading Speeds and Proofs of Theorems 2.2 and 2.3}

In this section, we explore the spreading speeds of \eqref{main-eq}, \eqref{nonlocal-eq}, and \eqref{discrete-eq},
and prove Theorems 2.2 and 2.3. Throughout this section, we assume ${\rm (H1)}$ and ${\rm (H2)}$.

We first prove  two lemmas.

\begin{lemma}
\label{technical-lm1} Let $\xi\in S^{N-1}$, $c>0$, $1\leq i\le
3$, and $u_0\in X_i^+$ be given.

\begin{itemize}
\vspace{-0.05in}\item[(1)] If $\liminf_{x\cdot\xi\leq ct,t\to\infty}u_i(t,x;u_0)>0$,
then for any $0<c^{'}<c$,
\vspace{-0.05in}$$
\limsup_{x\cdot\xi\leq c^{'}t,t\to\infty}|u_i(t,x;u_0)-u_i^*(x)|=0.
\vspace{-0.05in}$$

\vspace{-0.05in}\item[(2)] If $\liminf_{|x\cdot\xi|\leq
ct,t\to\infty}u_i(t,x;u_0)>0$, then for any $0<c^{'}<c$,
\vspace{-0.05in}$$
\limsup_{|x\cdot\xi|\leq c^{'}t,t\to\infty}|u_i(t,x;u_0)-u_i^*(x)|=0.
\vspace{-0.05in}$$

\vspace{-0.05in}\item[(3)] If $\liminf_{\|x\|\leq ct,t\to\infty}u_i(t,x;u_0)>0$,
then for any $0<c^{'}<c$,
\vspace{-0.05in}$$
\limsup_{\|x\|\leq c^{'}t,t\to\infty}|u_i(t,x;u_0)-u_i^*(x)|=0.
\vspace{-0.05in}$$
\end{itemize}
\end{lemma}

\begin{proof}
(1) Suppose that $\liminf_{x\cdot\xi\leq
ct,t\to\infty}u_i(t,x;u_0)>0$. Then there are $\delta$ and $T>0$
such that
\vspace{-0.05in}$$
u_i(t,x;u_0)\geq \delta\quad \forall
(t,x)\in\RR^+\times\mathcal{H}_i,\,\, x\cdot\xi\leq c t,\,\,
t\geq T.
\vspace{-0.05in}$$
Assume that the conclusion of (1) is not true. Then there are $0<c^{'}<c$, $\epsilon_0>0$,
$x_n\in\mathcal{H}_i$, and $t_n\in\RR^+$ with $x_n\cdot\xi\leq c^{'}t_n$
and $t_n\to\infty$ such that
\vspace{-0.05in}\begin{equation}
\label{tech-lm-eq1}
|u_i(t_n,x_n;u_0)-u_i^*(x_n)|\geq \epsilon_0\quad \forall n\ge 1.
\vspace{-0.05in}\end{equation}
Without loss of generality, we may assume that $x_n\to x^*$ as
$n\to\infty$ in the case that $\{\|x_n\|\}$ is bounded (this implies
that $f_i(x+x_n,u)\to f_i(x+x^*,u)$ uniformly in $(x,u)$ in bounded
sets) and $f_i(x+x_n,u)\to f_i^0(u)$ as $n\to\infty$ uniformly in
$(x,u)$ on bounded sets in the case that $\{\|x_n\|\}$ is unbounded.

Let $\tilde u_0\in X_i^+$,
\vspace{-0.05in}$$
\tilde u_0(x)=\delta\quad \forall x\in\mathcal{H}_i.
\vspace{-0.05in}$$
By Theorem 2.1, there is $\tilde T>0$ such that
\vspace{-0.05in}\begin{equation}
\label{tech-lm-eq2}
u_i(\tilde T,x;u_0)-u_i^*(x)<\epsilon_0\quad \forall x\in\mathcal{H}_i,
\vspace{-0.05in}\end{equation}
\vspace{-0.05in}\begin{equation}
\label{tech-lm-eq3}
|u_i(\tilde T,x;\tilde
u_0,f_i(\cdot+x^*,\cdot))-u_i^*(x+x^*)|<\frac{\epsilon_0}{2},
\vspace{-0.05in}\end{equation}
and
\vspace{-0.05in}\begin{equation}
\label{tech-lm-eq4}
|u_i(\tilde T,x;\tilde u_0,f_i^0)-u_i^0|<\frac{\epsilon_0}{2}.
\vspace{-0.05in}\end{equation}

Without loss of generality, we may assume that $t_n-\tilde T\geq T$ for $n\geq 1$.
Let $\tilde u_{0n}\in X_i^+$ be such that $\tilde u_{0n}(x)=\delta$ for
$x\cdot\xi\leq \frac{c^{'}+c}{2}(t_n-\tilde T)$, $0\le \tilde u_{0n}(x)\le
\delta$ for $\frac{c^{'}+c}{2}(t_n-\tilde T)\le x\cdot\xi\le c
(t_n-\tilde T)$, and $\tilde u_{0n}(x)=0$ for $x\cdot\xi\geq c(t_n-\tilde T)$.
Then
\vspace{-0.05in}$$
u_i(t_n-\tilde T,\cdot;u_0)\geq \tilde u_{0n}(\cdot)
\vspace{-0.05in}$$
and hence
\vspace{-0.05in}\begin{align}
\label{tech-lm-eq5}
u_i(t_n,x_n;u_0)&=u_i(\tilde T,x_n;u_i(t_n-\tilde T,\cdot;u_0))\nonumber\\
&=u_i(\tilde T,0;u_i(t_n-\tilde T,\cdot+x_n;u_0),f_i(\cdot+x_n,\cdot))\nonumber\\
&\geq u_i(\tilde T,0;\tilde u_{0n}(\cdot+x_n),f_i(\cdot+x_n,\cdot)).
\vspace{-0.05in}\end{align}

Observe that $\tilde u_{0n}(x+x_n)\to \tilde u_0$ as $n\to\infty$
uniformly in $x$  on bounded sets.
In the case  that $f_i(x+x_n,u)\to f_i^0(u)$,  by Proposition \ref{basic-convergence},
\vspace{-0.05in}$$
u_i(\tilde T,0;\tilde u_{0n}(\cdot+x_n),f_i(\cdot+x_n,\cdot))\to
u_i(\tilde T,0;\tilde u_0,f_i^0(\cdot))
\vspace{-0.05in}$$
as $n\to\infty$.
By \eqref{tech-lm-eq4} and \eqref{tech-lm-eq5},
\vspace{-0.05in}\begin{equation}
\label{tech-lm-eq6}
u_i(t_n,x_n;u_0)>u_i^0-\epsilon_0/2\quad {\rm for}\quad n\gg 1.
\vspace{-0.05in}\end{equation}
By Lemma \ref{tech-lm2},
\vspace{-0.05in}\begin{equation}
\label{tech-lm-eq7}
u_i^0>u_i^*(x_n)-\epsilon_0/2\quad {\rm for}\quad n\gg 1.
\vspace{-0.05in}\end{equation}
By \eqref{tech-lm-eq2}, \eqref{tech-lm-eq6}, and \eqref{tech-lm-eq7},
\vspace{-0.05in}$$
|u_i(t_n,x_n;u_0)-u_i^*(x_n)|<\epsilon_0\quad {\rm for}\quad n\gg 1.
\vspace{-0.05in}$$
This contradicts to \eqref{tech-lm-eq1}.

In the case that  $x_n\to x^*$, by Proposition \ref{basic-convergence} again,
\vspace{-0.05in}$$
u_i(\tilde T,0;\tilde u_{0n}(\cdot+x_n),f_i(\cdot+x_n,\cdot))\to
u_i(\tilde T,0;\tilde u_0,f_i(\cdot+x^*,\cdot))
\vspace{-0.05in}$$
as $n\to\infty$.
By \eqref{tech-lm-eq3} and \eqref{tech-lm-eq5},
\vspace{-0.05in}\begin{equation}
\label{tech-lm-eq8}
u_i(t_n,x_n;u_0)>u_i^*(x^*)-\epsilon_0/2\quad {\rm for}\quad n\gg
1.
\vspace{-0.05in}\end{equation}
By the continuity of $u_i^*(\cdot)$,
\vspace{-0.05in}\begin{equation}
\label{tech-lm-eq9} u_i^*(x^*)>u_i^*(x_n)-\epsilon_{0}/2\quad {\rm
for}\quad n\gg 1.
\vspace{-0.05in}\end{equation}
By \eqref{tech-lm-eq2}, \eqref{tech-lm-eq8}, and \eqref{tech-lm-eq9},
\vspace{-0.05in}$$
|u_i(t_n,x_n;u_0)-u_i^*(x_n)|<\epsilon_0\quad {\rm for}\quad n\gg 1.
\vspace{-0.05in}$$
This contradicts to \eqref{tech-lm-eq1} again.

 Hence
\vspace{-0.05in}$$
\lim_{x\cdot\xi\leq c^{'}t,t\to\infty}|u_i(t,x;u_0)-u_i^*(x)|=0
\vspace{-0.05in}$$
for all $0<c^{'}<c$.

(2) It can be proved by the similar arguments as in (1).

(3) It can also be proved by the similar arguments as in (1).
\end{proof}

\begin{lemma}
\label{technical-lm2}
Let $M>0$ be such that $f_i(x,u)<0$ for $x\in \mathcal{H}_i$, $u\in[0,M]$, and $i=1,2,3$. Then for any $\epsilon>0$,
there are $p\in\NN^N$ and $g_i:\mathcal{H}_i\times\RR\to\RR$ such that for any $u\in\RR$, $g_i(\cdot,u)\in X_{i,p}$,
$g_i(\cdot,\cdot)$ satisfies (P1) and (P3), and
\vspace{-0.05in}$$
f_i(x,u)\ge g_i(x,u)\quad \forall x\in\mathcal{H}_i,\,\, u\in [0,M],
\vspace{-0.05in}$$
\vspace{-0.05in}$$
\hat g_i(0)\geq f_i^0(0)-\epsilon,
\vspace{-0.05in}$$
where $\hat g_i(\cdot)$ is as in \eqref{average-g-eq} $(i=1,2,3)$.
\end{lemma}

\begin{proof}
By   Lemma \ref{tech-lm1}, for any $\epsilon>0$, there are $p\in\NN^N$ and $h_i(\cdot)\in X_{i,p}\cap C^N(\mathcal{H}_i,\RR)$ such that
\vspace{-0.05in}$$
f_i(x,0)\geq h_i(x)\,\, \forall x\in\mathcal{H}_i
\quad {\rm and}\quad \hat h_i\geq f_i^0(0)-\epsilon
\vspace{-0.05in}$$
for $i=1,2,3$.
 Fix $1\le i\le 3$ and choose $M_i>0$ such that
\vspace{-0.05in}$$
f_i(x,u)\ge g_i(x,u):=h_i(x)-M_i u\quad {\rm for}\quad x\in\mathcal{H}_i,\,\,  0\leq u\leq M.
\vspace{-0.05in}$$
It is not difficult to see that $g_i(\cdot,\cdot)$  $(1\le i\le 3$) satisfy the lemma.
\end{proof}

In the following, $c_1^0(\xi)$, $c_2^0(\xi)$, and $c_3^0(\xi)$ are as in \eqref{c-1-0-eq}, \eqref{c-2-0-eq},
and \eqref{c-3-0-eq}, respectively ($\xi\in S^{N-1}$).
Observe that $\lambda_i(\mu,\xi,f_i^0(0))$ $(i=1,2,3$) exist and 
\vspace{-0.05in}\begin{equation*}
\begin{cases}
\lambda_1(\mu,\xi,f_1^0(0))=f_1^0(0)+\mu^2\cr
\lambda_2(\mu,\xi,f_2^0(0))=\int_{\RR^N}e^{-\mu z\cdot\xi}\kappa(z)dz-1+f_2^0(0)\cr
\lambda_3(\mu,\xi,f_3^0(0))=\sum_{k\in K}a_k(e^{-\mu k\cdot\xi}-1)+f_3^0(0).
\end{cases}
\vspace{-0.05in}\end{equation*}
If no confusion occurs, we may denote $\lambda_i(\mu,\xi,f_i^0(0))$ by
$\lambda_i(\mu,\xi)$ ($i=1,2,3$).
Observe also that $v_1(t,x)=e^{-\mu (x\cdot\xi-\frac{\lambda_1(\mu,\xi)}{\mu}t)}$,
$v_2(t,x)=e^{-\mu (x\cdot\xi-\frac{\lambda_2(\mu,\xi)}{\mu}t)}$, and
$v_3(t,j)=e^{-\mu (j\cdot\xi-\frac{\lambda_3(\mu,\xi)}{\mu}t)}$ are solutions of
\vspace{-0.05in}\begin{equation}
\label{aux-main-linear}
v_t(t,x)=\Delta v(t,x)+f_1^0(0)v(t,x),\quad x\in\RR^N,
\vspace{-0.05in}\end{equation}
\vspace{-0.05in}\begin{equation}
\label{aux-nonlocal-linear}
v_t(t,x)=\int_{\RR^N}\kappa(y-x)v(t,y)dy-v(t,x)+f_2^0(0)v(t,x),\quad x\in\RR^N,
\vspace{-0.05in}\end{equation}
and
\vspace{-0.05in}\begin{equation}
\label{aux-discrete-linear}
v_t(t,j)=\sum_{k\in K}a_k(v(t,j+k)-v(t,j))+f_3^0(0)v(t,j),\quad j\in\ZZ^N,
\vspace{-0.05in}\end{equation}
respectively.

\begin{proof}[Proof of Theorem 2.2]
Fix $\xi\in S^{N-1}$ and $1\le i\le 3$.
We first prove that for any $c^{'}>c_i^0(\xi)$ and $u_0\in X_i^+(\xi)$,
\vspace{-0.05in}\begin{equation}
\label{spreading-eq1}
\limsup_{x\cdot\xi\geq c^{'}t, t\to\infty} u_i(t,x;u_0)=0.
\vspace{-0.05in}\end{equation}
To this end,
take a $c$ such that $c^{'}>c>c_i^*(\xi)$.
Note that there is $\mu_i^*>0$ such that
\vspace{-0.05in}$$
c_i^0(\xi)=\frac{\lambda_i(\xi,\mu^*_i)}{\mu^*_i}
\vspace{-0.05in}$$
and there is $\mu\in (0,\mu^*_i)$ such that
\vspace{-0.05in}$$
c=\frac{\lambda_i(\mu,\xi)}{\mu}.
\vspace{-0.05in}$$
Take $d>M>0$  such that
\vspace{-0.05in}$$
u_0(x)\leq M\quad {\rm and}\quad  u_0(x)\leq d e^{-\mu x\cdot\xi}\quad \forall x\in\mathcal{H}_i,
\vspace{-0.05in}$$
\vspace{-0.05in}\begin{equation}
\label{thm2-2-eq1}
f_i(x,M)<0\quad \forall x\in\mathcal{H}_i,
\vspace{-0.05in}\end{equation}
and
\vspace{-0.05in}\begin{equation}
\label{thm2-2-eq2}
f_i(x,u)=f_i^0(u)\quad {\rm for}\quad x\cdot\xi\geq -\frac{1}{\mu}\ln \frac{M}{d} (>0).
\vspace{-0.05in}\end{equation}
Observe that by \eqref{thm2-2-eq1} and ${\rm (H1)}$, for $(t,x)\in (0,\infty)\times \mathcal{H}_i$ with $d e^{-\mu(x\cdot\xi-ct)}\geq M$, i.e.,
$x\cdot\xi\leq -\frac{1}{\mu}\ln \frac{M}{d}+ct$,
\vspace{-0.05in}$$f_i(x,de^{-\mu(x\cdot\xi-ct)})<0<f_i^0(0).
\vspace{-0.05in}$$
 By \eqref{thm2-2-eq2}, for $(t,x)\in (0,\infty)\times \mathcal{H}_i$
with $d e^{-\mu(x\cdot\xi-ct)}\le M$, i.e, $x\cdot\xi\ge  -\frac{1}{\mu}\ln \frac{M}{d}+ct$,
\vspace{-0.05in}$$
f_i(x,de^{-\mu (x\cdot\xi-ct)})=f_i^0(de^{-\mu (x\cdot\xi-ct)})\leq f_i^0(0).
\vspace{-0.05in}$$
It then follows that $u=d e^{-\mu(x\cdot\xi-ct)}$, which is a solution of \eqref{aux-main-linear} or \eqref{aux-nonlocal-linear}
or \eqref{aux-discrete-linear} if $i=1$ or $2$ or $3$,
 is  a super-solution of (1.i) and hence by Proposition \ref{basic-comparison},
\vspace{-0.05in}\begin{equation}
\label{thm2-2-eq3}
u_i(t,x;u_0)\leq de^{-\mu(x\cdot\xi-ct)}\quad \forall t>0\,\, x\in\mathcal{H}_i.
\vspace{-0.05in}\end{equation}
This implies that \eqref{spreading-eq1} holds.

Next, we prove that for any $c^{'}<c_i^0(\xi)$ and any $u_0\in X_i^+(\xi)$,
\vspace{-0.05in}\begin{equation}
\label{spreading-eq2}
\limsup_{x\cdot\xi\leq c^{'}t,t\to\infty}|u_i(t,x;u_0)-u_i^*(x)|=0.
\vspace{-0.05in}\end{equation}
To this end, take a $c\in\RR$ such that $c^{'}<c<c_i^0(\xi)$.
Let $M>0$ be such that $u_0(x)\leq M$ and $f_i(x,M)<0$ for all $x\in\mathcal{H}_i$.
Then $u\equiv M$ is a super-solution of (1.i) and
\vspace{-0.05in}$$
u_i(t,x;u_0)\le M\quad \forall t\geq 0,\, \, x\in\mathcal{H}_i.
\vspace{-0.05in}$$
For any $\epsilon>0$, let $g_i(\cdot,\cdot)$ be as in  Lemma \ref{technical-lm2}.
By Proposition \ref{periodic-prop3}, for $\epsilon>0$ sufficiently small,
\vspace{-0.05in}$$
c_i^*(\xi,g_i(\cdot,\cdot))\geq c_i^*(\xi,\hat g_i(\cdot))>c.
\vspace{-0.05in}$$
By Propositions \ref{basic-comparison} and \ref{periodic-prop2},
\vspace{-0.05in}$$
\liminf_{x\cdot\xi\le ct,t\to\infty} u_i(t,x;u_0)\ge \liminf_{x\cdot \xi\le ct,t\to\infty}u_i(t,x;u_0,g_i)>0.
\vspace{-0.05in}$$
\eqref{spreading-eq2} then follows from Lemma \ref{technical-lm1}.

By \eqref{spreading-eq1} and \eqref{spreading-eq2}, $c_i^*(\xi)$ exists and $c_i^*(\xi)=c_i^0(\xi)$
for $i=1,2,3$. Moreover, \eqref{spreading-eq0} holds
\end{proof}

\begin{proof}[Proof of Theorem 2.3]
(1) It can be proved by similar arguments in \cite[Theorem D(1)]{ShZh1}. For completeness, we provide a proof in the following.

Fix $\xi\in S^{N-1}$ and $1\le i\le 3$.
 Let $ u_0\in X_i^+$ satisfy that  $u_0(x)=0$ for $x\in\mathcal{H}_i$ with $|x\cdot\xi|\gg 1$.
 Then there are $u_0^+\in X_i^+(\xi)$ and $u_0^-\in X_i^+(-\xi)$ such that
\vspace{-0.05in} $$
 u_0(x)\leq u_0^\pm(x)\quad \forall x\in\mathcal{H}_i.
\vspace{-0.05in} $$
 By Proposition \ref{basic-comparison} and Theorem 2.2,
\vspace{-0.05in} $$
 \limsup_{x\cdot\xi\ge c^{'}t,t\to\infty}u_i(t,x;u_0)\leq \limsup_{x\cdot\xi\ge c^{'}t,t\to\infty}u_i(t,x;u_i^+)=0\quad \forall c^{'}>c_i^*(\xi)
\vspace{-0.05in} $$
and
\vspace{-0.05in} $$
 \limsup_{x\cdot(-\xi)\ge c^{'}t,t\to\infty}u_i(t,x;u_0)\leq \limsup_{x\cdot(-\xi)\ge c^{'}t,t\to\infty}u_i(t,x;u_i^-)=0\quad \forall c^{'}>c_i^*(-\xi)
 \vspace{-0.05in}$$
 It then follows that
 \vspace{-0.05in}$$
 \limsup_{|x\cdot\xi|\ge c^{'}t,t\to\infty}u_i(t,x;u_0)=0\quad \forall c^{'}>\max\{c_i^*(\xi),c_i^*(-\xi)\}.
 \vspace{-0.05in}$$

(2) Fix $\xi\in S^{N-1}$ and $1\le i\le 3$. For given $0<c^{'}<\min\{c_i^*(\xi),c_i^*(-\xi)\}$, take a $c>0$ such that
$c^{'}<c<\min\{c_i^*(\xi),c_i^*(-\xi)\}$. For given $u_0\in X_i^+ $ satisfying the condition in Theorem 2.3 (2), let $M>0$ be such that $u_0(x)\leq M$ and $f_i(x,M)<0$ for all $x\in\mathcal{H}_i$.
Then $u\equiv M$ is a super-solution of (1.i) and
\vspace{-0.05in}$$
u_i(t,x;u_0)\le M\quad \forall t\geq 0,\, \, x\in\mathcal{H}_i.
\vspace{-0.05in}$$
For any $\epsilon>0$, let $g_i(\cdot,\cdot)$ be as in  Lemma \ref{technical-lm2}.
By Proposition \ref{periodic-prop3}, for $\epsilon>0$ sufficiently small,
\vspace{-0.05in}$$
c_i^*(\xi,g_i(\cdot,\cdot))\geq c_i^*(\xi,\hat g_i(\cdot))>c.
\vspace{-0.05in}$$
By Propositions \ref{basic-comparison} and \ref{periodic-prop2},
\vspace{-0.05in}$$
\liminf_{|x\cdot\xi|\leq ct, t\to\infty} u_i(t,x;u_0)\geq \liminf_{|x\cdot\xi|\le ct,t\to\infty}u_i(t,x;u_0,g_i)>0.
\vspace{-0.05in}$$
It then follows from Lemma \ref{technical-lm1} that
\vspace{-0.05in}$$
\limsup_{|x\cdot\xi|\leq c^{'}t,t\to\infty}|u_i(t,x;u_0)-u_i^*(x)|=0.
\vspace{-0.05in}$$

(3) It can be proved by similar arguments as in \cite[Theorem E (1)]{ShZh1}.
For completeness again, we provide a proof in the following.

Fix $\xi\in S^{N-1}$ and $1\le i\le 3$.
 Let $c>\sup_{\xi\in S^{N-1}}c_i^*(\xi)$.
Let $u_0\in X_i^+$ be such that $u_0(x)=0$ for $\|x\|\gg 1$.
 Note that for every given $\xi\in S^{N-1}$, there is $\tilde u_0(\cdot;\xi)\in X_i^+(\xi)$ such that
$u_0(\cdot)\leq \tilde u_0(\cdot;\xi)$. By Proposition \ref{basic-comparison},
\vspace{-0.05in}$$
0\leq u_i(t,x;u_0)\leq u_i(t,x;\tilde u_0(\cdot;\xi))
\vspace{-0.05in}$$
for $t>0$ and $x\in\mathcal{H}_i$. It then follows from  Theorem 2.2 that
\vspace{-0.05in}$$
0\leq \limsup_{x\cdot\xi\geq ct,t\to\infty}u_i(t,x;u_0)\leq \limsup_{x\cdot\xi\geq ct,t\to\infty}u_i(t,x;\tilde
u_0(\cdot;\xi))=0.
\vspace{-0.05in}$$

Take any $c^{'}>c$. Consider all $x\in\mathcal{H}_i$ with $\|x\|=c^{'}$. By the compactness of
$\p B(0,c^{'})=\{x\in\mathcal{H}_i|\,\|x\|=c^{'}\}$, there are $\xi_1,\xi_2,\cdots,\xi_{L}\in S^{N-1}$ such that for every
$x\in \p B(0,c^{'})$, there is $l$ ($1\leq l\leq L$) such that $x\cdot \xi_l\geq c$. Hence for every $x\in\mathcal{H}_i$ with
$\|x\|\geq c^{'}t$, there is $1\leq l\leq L$ such that $x\cdot
\xi_l=\frac{\|x\|}{c^{'}}\Bigl(\frac{c^{'}}{\|x\|}x\Bigr)\cdot \xi_l\geq \frac{\|x\|}{c^{'}}c\geq ct$. By the
above arguments,
\vspace{-0.05in}$$
0\leq \limsup_{x\cdot\xi_l\geq ct,t\to\infty}u_i(t,x;u_0)\leq \limsup_{x\cdot\xi_l\geq
ct,t\to\infty}u_i(t,x;\tilde u_0(\cdot;\xi_l))=0
\vspace{-0.05in}$$
for $l=1,2,\cdots L$.
 This implies that
\vspace{-0.05in}$$
\limsup_{\|x\|\geq c^{'}t, t\to\infty}u_i(t,x;u_0)=0.
\vspace{-0.05in}$$
 Since $c^{'}>c$ and $c>\sup_{\xi\in
S^{N-1}}c_i^*(\xi)$ are arbitrary, we have that for $c>\sup_{\xi\in S^{N-1}}c_i^*(\xi)$,
\vspace{-0.05in}$$
\limsup_{\|x\|\geq ct, t\to\infty}u_i(t,x;u_0)=0.
\vspace{-0.05in}$$

(4) It can be proved by similar arguments as in (2). To be more precise,
for given $0<c^{'}<\min\{c_i^*(\xi)\,|\, \xi\in S^{N-1}\}$, take a $c>0$ such that
$c^{'}<c<\min\{c_i^*(\xi)\,|\,\xi\in S^{N-1}\}$.
 For given $u_0\in $ satisfying the condition in Theorem 2.3 (4), let $M>0$ be such that $u_0(x)\leq M$ and $f_i(x,M)<0$ for all $x\in\mathcal{H}_i$.
Then $u\equiv M$ is a super-solution of (1.i) and
\vspace{-0.05in}$$
u_i(t,x;u_0)\le M\quad \forall t\geq 0,\, \, x\in\mathcal{H}_i.
\vspace{-0.05in}$$
For any $\epsilon>0$, let $g_i(\cdot,\cdot)$ be as in  Lemma \ref{technical-lm2}.
By Proposition \ref{periodic-prop3}, for $\epsilon>0$ sufficiently small,
\vspace{-0.05in}$$
c_i^*(\xi,g_i(\cdot,\cdot))\geq c_i^*(\xi,\hat g_i(\cdot))>c.
\vspace{-0.05in}$$
By Propositions \ref{basic-comparison} and \ref{periodic-prop2},
\vspace{-0.05in}$$
\liminf_{\|x\|\leq ct, t\to\infty} u_i(t,x;u_0)\ge \liminf_{\|x\|\le ct,t\to\infty}u_i(t,x;u_0,g_i)>0.
\vspace{-0.05in}$$
It then follows from Lemma \ref{technical-lm1} that
\vspace{-0.05in}$$
\limsup_{\|x\|\leq c^{'}t,t\to\infty}|u_i(t,x;u_0)-u_i^*(x)|=0.
\vspace{-0.05in}$$
\end{proof}

\end{document}